\documentclass[11pt]{amsart}

\newif\ifdraft\draftfalse

\usepackage{amssymb}
\usepackage{amsthm}
\usepackage{eucal}
\usepackage[english]{babel}
\usepackage{nicefrac}
\usepackage{times}
\usepackage{latexsym,amscd,amsmath,amsthm,amssymb}

\makeatletter
\def\@begintheorem#1#2[#3]{%
    \def\naam{#1}
  \deferred@thm@head{\the\thm@headfont \thm@indent
    \@ifempty{#1}{\let\thmname\@gobble}{\let\thmname\@iden}%
    \@ifempty{#2}{\let\thmnumber\@gobble}{\let\thmnumber\@iden}%
    \@ifempty{#3}{\let\thmnote\@gobble}{\let\thmnote\@iden}%
    \thm@swap\swappedhead\thmhead{#1}{#2}{#3}%
    \the\thm@headpunct
    \thmheadnl 
    \hskip\thm@headsep
  }%
  \ignorespaces}
\makeatother

\newcommand{\kantlijndraft}[1]{\ifdraft\hspace{-\lastskip}%
\vadjust{\vspace{-1mm}\smash{\llap{{\tt #1}\hspace{8mm}}}\vspace{1mm}}\fi}

\def\voegToe#1#2#3{\immediate\write1{\string\newlabel{#1}{{#2}{#3}}}}

\newcommand{\thlabel}[1]{\voegToe{#1}{\naam\noexpand~\thetheorem}{\thepage}\kantlijndraft{#1}}

\makeatletter
\renewcommand{\label}[1]{\voegToe{#1}{\@currentlabel}{\thepage}\kantlijndraft{#1}}
\makeatother

\newtheorem{theorem}{Theorem}[section]
\newtheorem{lemma}[theorem]{Lemma}
\newtheorem{corollary}[theorem]{Corollary}
\newtheorem{question}[theorem]{Question}

\newtheorem{proposition}[theorem]{Proposition}
\theoremstyle{definition}
\newtheorem{example}[theorem]{Example}
\newtheorem{definition}[theorem]{Definition}
\theoremstyle{remark}
\newtheorem*{remark}{Remark}
\newtheorem*{notation}{Notation}
\numberwithin{equation}{section}

\newtheorem{claim2}{\sc Claim}



\newcommand{\sse}{\subseteq}						
\newcommand{\minus}{\backslash}						
\newcommand{\Un}{\bigcup}							
\newcommand{\un}{\cup}								
\newcommand{\Meet}{\bigcap}							
\newcommand{\meet}{\cap}							
\newcommand{\es}{\varnothing}						

\newcommand{\scr}[1]{\ensuremath{\mathcal{#1}}}

\newcommand{\al}{\alpha}

\newcommand{\ka}{\kappa}

\newcommand{\pp}{\prime\prime}

\def\cprime{$'$}

\def\sapirovskii{{\v{S}}apirovski{\u\i}}
\def\arhangelskii{Arhangel{\cprime}ski{\u\i}}
\def\katetov{Kat\v{e}tov}
\def\juhasz{Juh{\'a}sz}

\begin{document}

\title{On the cardinality of Hausdorff spaces and H-closed spaces}

\author{N.A. Carlson}\address{Department of Mathematics, California Lutheran University, 60 W. Olsen Rd, MC 3750, Thousand Oaks, CA 91360 USA}
\email{ncarlson@callutheran.edu}

\author{J.R. Porter}
\address{Department of Mathematics, University of Kansas, 405 Snow Hall, 1460 Jayhawk Blvd, Lawrence, KS 66045-7523, USA}
\email{porter@ku.edu}



\begin{abstract} 
We introduce the cardinal invariant $aL^\prime(X)$ and show that $|X|\leq 2^{aL^\prime(X)\chi(X)}$ for any Hausdorff space $X$ (a corollary of Theorem~\ref{newbound}). This invariant has the properties a) $aL^\prime(X)=\aleph_0$ if $X$ is H-closed, and b) $aL(X)\leq aL^\prime(X)\leq aL_c(X)$. Theorem~\ref{newbound} then gives a new improvement of the well-known Hausdorff bound $2^{L(X)\chi(X)}$ from which it follows that $|X|\leq 2^{\psi_c(X)}$ if $X$ is H-closed (Dow/Porter~\cite{DowPorter82}). The invariant $aL^\prime(X)$ is constructed using convergent open ultrafilters and an operator $c:\scr{P}(X)\to\scr{P}(X)$ with the property $clA\sse c(A)\sse cl_\theta(A)$ for all $A\sse X$. As a comparison with this open ultrafilter approach, in $\S 3$ we additionally give a $\kappa$-filter variation of Hodel's proof~\cite{Hodel2006} of the Dow-Porter result. Finally, for an infinite cardinal $\kappa$, in $\S 5$ we introduce $\kappa$wH-closed spaces, $\kappa H^\prime$-closed spaces, and $\kappa H^{\prime\prime}$-closed spaces. The first two notions generalize the H-closed property. Key results in this connection are that a) if $\kappa$ is an infinite cardinal and $X$ a $\kappa$wH-closed space with a dense set of isolated points such that $\chi(X)\leq\kappa$, then $|X|\leq 2^{\kappa}$, and b) if $X$ is $\kappa H^\prime$-closed or $\kappa H^{\prime\prime}$-closed then $aL^\prime(X)\leq\kappa$. This latter result relates these notions to the invariant $aL^\prime(X)$ and the operator $c$.
\end{abstract}

\maketitle

\section{Introduction.}
A space $X$ is \emph{H-closed} if every open cover $\scr{V}$ of $X$ has a finite subfamily $\scr{W}$ such that $X=\Un_{W\in\scr{W}}clW$. In 1982 Dow and Porter \cite{DowPorter82} used H-closed extensions of discrete spaces to demonstrate that $|X|\leq 2^{\chi(X)}$ (in fact, $|X|\leq 2^{\psi_c(X)}$) for any H-closed space $X$. The technique was simplified in Porter~\cite{por93}, and in 2006 Hodel~\cite{Hodel2006} gave a proof of the Dow-Porter result using $\kappa$-nets. 

A natural general question is the following:

\begin{question} Does there exists a strengthening of \arhangelskii's cardinal inequality $|X|\leq 2^{L(X)\chi(X)}$ \cite{arh1969} for a general Hausdorff space $X$ for which it follows as a corollary that $|X|\leq 2^{\chi(X)}$ if $X$ is H-closed?
\end{question}

This question was asked by Angelo Bella in personal communication with the second author. Another way to ask this is, does there exists a property $\scr{P}$ of a Hausdorff space $X$ that a) generalizes both the H-closed property and the Lindel\"of property simultaneously, and b) $|X|\leq 2^{\chi(X)}$ for spaces $X$ with property $\scr{P}$?

As both H-closed spaces and Lindel\"of spaces are \emph{almost-Lindel\"of} (that is, every open cover has a countable subfamily whose closures cover), the property ``almost-Lindel\"of'' would seem to be a suitable candidate. However, in 1998 Bella and Yaschenko \cite{bey1998} showed that if $\kappa$ is a non-measurable cardinal then there exists an almost-Lindel\"of, first-countable Hausdorff space $X$ such that $|X|>\kappa$. Thus, $|X|\leq 2^{\chi(X)}$ does not hold for all almost-Lindel\"of Hausdorff spaces $X$. (It does, however, hold if the space is additionally Urysohn \cite{BellaCammaroto1988}). In 1988, Bella and Cammaroto \cite{BellaCammaroto1988} gave the bound $2^{aL_c(X)t(X)\psi_c(X)}$ for the cardinality of a Hausdorff space $X$, where $aL_c(X)$ is defined before Definition~\ref{alprime} below. As $aL_c(X)\leq L(X)$, this suggests that ``$aL_c(X)=\aleph_0$'' might be the required property $\scr{P}$. Yet the \katetov~H-closed extension $\kappa\omega$ of the discrete space $\omega$ is an example of an H-closed space for which $aL_c(X)=\mathfrak{c}>\aleph_0$, demonstrating that the property $aL_c(X)=\aleph_0$ does not hold for all H-closed spaces.

In this study we construct a cardinal invariant $aL^\prime(X)$ such that a) $|X|\leq 2^{aL^\prime(X)\chi(X)}$ for a Hausdorff space $X$ (Theorem~\ref{newbound} gives a slightly stronger version of this statement), b) $aL(X)\leq aL^\prime(X)\leq aL_c(X)$ (Proposition~\ref{anotherProp}), and c) $aL^\prime(X)=\aleph_0$ if $X$ is H-closed (follows from Corollary~\ref{hclosed}). Thus, the property ``$aL^\prime(X)=\aleph_0$'' is the required property $\scr{P}$ above. Theorem~\ref{newbound} then gives a new bound on the cardinality of a Hausdorff space that is strong enough to capture the H-closed bound $2^{\chi(X)}$ given by Dow and Porter.

For an open set $U$ in a space $X$, convergent open ultrafilters are used to define a set $\widehat{U}$  (Definition~\ref{hatdef}) such that $U\sse\widehat{U}\sse clU$. Using the set $\widehat{U}$, we then define an operator $c:\scr{P}(X)\to\scr{P}(X)$ that satisfies $clA\sse c(A)\sse cl_\theta(A)$ for all $A\sse X$. Both $\widehat{U}$ and the function $c$ have relationships to the Iliadis absolute $EX$ that are outlined in $\S 2$. After this set-up, the invariant $aL^\prime(X)$ is defined as in Definition~\ref{alprime}.

In Theorem~\ref{HclosedChar}, we give the following characterization of H-closed spaces, which is of interest in its own right: a space $X$ is H-closed if and only if for every open cover $\scr{V}$ of X there exists $\scr{W}\in[\scr{V}]^{<\omega}$ such that $X=\Un_{W\in\scr{W}}\widehat{W}$. Given that $\widehat{W}\sse clW$ for every open set $W$, this characterization is then a logically stronger property that the usual definition of the H-closed property. It follows naturally to define a cardinal invariant $L^\prime(X)$ as the least infinite cardinal $\kappa$ such that if $\scr{V}$ is a cover of $X$ then there exists $\scr{W}\in[\scr{V}]^{\leq\kappa}$ such that $X=\Un_{W\in\scr{W}}\widehat{W}$ (Definition~\ref{lprime}). Theorem~\ref{HclosedChar} shows that $L^\prime(X)=\aleph_0$ if $X$ is H-closed. We demonstrate that $L^\prime(X)$ is hereditary on $c$-closed subsets (Proposition~\ref{Lhered}), from which it follows that $aL^\prime(X)\leq L^\prime(X)\leq L(X)$. Thus, given our main cardinality bound for general Hausdorff spaces (Theorem~\ref{newbound}), we see that the following is a sufficient property of both Lindel\"of and H-closed spaces $X$ from which it follows that $|X|\leq 2^{\chi(X)}$: every open cover $\scr{V}$ has a countable subfamily $\scr{W}$ such that $X=\Un_{W\in\scr{W}}\widehat{W}$ (that is, $L^\prime(X)=\aleph_0)$. As $aL^\prime(X)\leq L^\prime(X)$, another such property (albeit weaker) is ``$aL^\prime(X)=\aleph_0$'', as mentioned above.

In~\cite{Hodel2006}, Hodel gave a proof that $|X|\leq 2^{\psi_c(X)}$ for H-closed spaces $X$ using the notion of a $\kappa$-net for a cardinal $\kappa$. This proof is different than previous proofs of this result given by Dow and Porter, and also different than the approach taken in Theorem~\ref{newbound} in this study. In $\S3$ we use a filter characterization of H-closed spaces given in Theorem~\ref{H-cldChar} and the $c$-adherence of a filter to give another proof that the cardinality of an H-closed space $X$ is bounded by $2^{\psi_c(X)}$. This particular method can be seen as a variation of the method used by Hodel ~\cite{Hodel2006} for nets. We present two examples at the end of $\S3$.

In $\S 4$ we give the proof of our main result, Theorem~\ref{newbound}, after establishing preliminary results in $\S2-\S4$. The proof is fundamentally a standard closing-off argument. We use a theorem of Hodel (re-stated in Theorem~\ref{hodel}) that gives a set-theoretic generalization of many such arguments. Typically the closure operator is used in a closing-off argument, or occassionally the $\theta$-closure operator. We use the operator $c$ referred to above.

In $\S5$ we introduce two notions that generalize the H-closed property and a related third notion. The first is, for an infinite cardinal $\kappa$, the concept of a $\kappa$wH-closed space (Definition~\ref{kwHclosed}). This notion grows naturally out of recent work of Osipov in~\cite{O}. In Proposition~\ref{propA} we give this characterization of H-closed: $X$ is H-closed  if an only if $X$ is $\aleph_0$wH-closed. A key result is Theorem~\ref{kwHclosedDense}, which states that if $\kappa$ is an infinite cardinal and $X$ is a $\kappa$wH-closed space with a dense set of isolated points such that $\chi(X)\leq\kappa$, then $|X|\leq 2^{\kappa}$. The second notion introduced in $\S5$ is that of a $\kappa H^\prime$-closed space (Definition~\ref{defII}). Proposition~\ref{HclII} demonstrates that $X$ is H-closed if and only if $X$ is $\aleph_0H^\prime$-closed. After defining $z(X)=inf\{\kappa\geq\aleph_0:X\textup{ is }\kappa H^\prime-\textup{closed}\}$, it is shown in Corollary~\ref{corII}(a) that $aL^\prime(X)^+\leq z(X)$, thereby relating the notion of $\kappa H^\prime$-closed to concepts defined in previous sections. It follows immediately $|X|\leq 2^{z(X)\chi(X)}$ for any Hausdorff space $X$ after applying Theorem~\ref{newbound}. Finally, we introduce the property of $\kappa H^{\pp}$-closed in Definition~\ref{kH''closed} and, for a space $X$, we define the cardinal invariant $z^\prime(X) = inf\{\kappa \geq \aleph_0: X \text{ is } \kappa\text{H}''\text{--closed}\}$. While it can be shown that $aL^\prime(X)\leq z^\prime(X)$ and thus $|X|\leq 2^{z^\prime(X)\chi(X)}$ for any Hausdorff space $X$ (Corollary~\ref{corIII}), it is not guaranteed that a $\aleph_0H^{\prime\prime}$-closed space is H-closed. In fact, any countable space is $\aleph_0H^{\prime\prime}$-closed.

\emph{All spaces are assumed to be Hausdorff}. For all undefined notions see Engelking~\cite{Engelking},  \juhasz~\cite{Juh}, or Porter-Woods~\cite{PW}. Hodel's survey paper~\cite{Hodel2006} also contains thorough discussion of many cardinal invariants and cardinality bounds related to those discussed in this study. 

\section{Construction of the cardinal function $aL^\prime(X)$.}
Given a Hausdorff space $X$ and an open set $U$ of $X$, define 
$$0U=\{\scr{U}:\scr{U}\textup{ is a convergent open ultrafilter containing }U\}.$$
We recall the construction of the Iliadis absolute $EX$ as the set of convergent open ultrafilters on $X$ with the topology generated by the basis $\{0U: U\textup{ is open in }X\}$. (See \cite{PW}, Chapter 6, for example). Under this topology $EX$ is an extremally disconnected, zero-dimensional, Tychonoff space. For each $\scr{U}\in EX$, let $k(\scr{U})$ be the unique convergent point of $\scr{U}$.  We have the following basic facts concerning $EX$ and the map $k:EX\to X$ (see \cite{PW} 6.6(e)(5), 6.8(d,f) and \cite{PV} 1.2(b)). Recall a subset $A$ of a space $X$ is an \emph{H-set} if for every cover $\scr{V}$ of $A$ by sets open in $X$ there exists $\scr{W}\in[\scr{V}]^{<\omega}$ such that $A\sse\Un_{W\in\scr{W}}clW$ and a space $X$ is Kat\v etov if $X$ has a coarser H-closed topology.

\begin{proposition}\label{ex}
For a open sets $U,V\sse X$ of a space $X$, and $\scr{U}\in EX$,
\begin{itemize}
\item[(a)] The map $k:EX\to X$ is a $\theta$-continuous, perfect, irreducible, surjection,
\item[(b)] $U \in \scr U$ iff $int(cl(U)) \in \scr U$ iff $\scr U \in 0U$, and thus $0U=0(int(clU))$,
\item[(c)] $k[0U] = cl(U)$,
\item[(d)] $k^{\leftarrow}(k(\scr U)) \subseteq 0U$ iff $k(\scr U) \in int(cl(U))$, and
\item[(e)] If $B \subseteq EX$ is compact, then $k[B]$ is Kat\v etov and  an H-set.
\item[(f)] $0(U\meet V)=0U\meet 0V$ and $0(U\un V)=0U\un0V$.
\end{itemize}
\end{proposition}

\noindent Let $b: X \rightarrow EX$ be an injective function such that $k \circ b = id_X$. That is, for all $x\in X$, $b(x)$ is an open ultrafilter converging to $x$. Denote the subspace $b[X]$ of $EX$ by $X_b$.  The space $X_b$ is a \emph{section} of $EX$ \cite{pw78} and is an extremally disconnected, Tychonoff space.  We observe that $k\left|_{X_b}\right.:X_b\to X$ is a bijection as $X$ is Hausdorff.

\begin{definition}\label{hatdef}
For a space $X$, an open set $U$, and a section $X_b$ of $EX$, define 
$$\widehat{U}_b = \{x \in X: U \in b(x)\}.$$
\end{definition}
We give several properties of $\widehat{U}_b$ in Proposition~\ref{hatu}. As is indicated in Proposition~\ref{hatu}(a), $\widehat{U}_b$ consists of a special set of closure points $x$ of $U$ having the stronger property that $U$ is a member of the open ultrafilter $b(x)$. The set $\widehat{U}_b$ will play a major role in the construction of the cardinal invariant $aL^\prime(X)$ and in the proof of our main theorem, Theorem~\ref{newbound}. In addition, for a space $(X,\tau)$, $\{\widehat{U}_b: U \in \tau(X)\}$ forms a basis for a topology $\sigma_b$ on $X$ such that $(X,\sigma_b)$ is homeomorphic to the section $X_b$ (Proposition~\ref{sectionTop}). Furthermore, Proposition~\ref{HclosedChar} gives a characterization of H-closed spaces using sets of the form $\widehat{U}_b$. It is this characterization that will give the cardinality bound $2^{\psi_c(X)}$ for H-closed spaces as an immediate consequence of the general Hausdorff bound given in Theorem~\ref{newbound}.

\begin{proposition}\label{hatu} 
Let $(X,\tau)$ be a space, $U,V$ open sets, and $X_b$ be any section of $EX$. Then,
\begin{itemize}
\item[(a)] $U\sse\widehat{U}_b\sse clU$,
\item[(b)] $\widehat{U}_b=k[0U \cap X_b]$,
\item[(c)] $\left(\widehat{U \cap V}\right)_b = \widehat{U}_b  \cap \widehat{V}_b$\space and $\left(\widehat{U \cup V}\right)_b = \widehat{U}_b  \cup \widehat{V}_b$\space, 
\item[(d)] $X\backslash \widehat{U} _b=(\widehat{X\backslash cl_XU})_b$\space.
\end{itemize}
\end{proposition}

\begin{proof}
(a) If $x\in U$, then $U$ is a member of any open ultrafilter converging to $x$. Thus, $U\in b(x)$ and $x\in\widehat{U}_b$. This shows $U\sse\widehat{U}_b$. Suppose $x\in\widehat{U}_b$. Then $U\in b(x)$. If $x\in X\minus clU$, then $X\minus clU\in b(x)$ and thus $\es=U\meet (X\minus clU)\in b(x)$, a contradiction. Thus $x\in clU$ and $\widehat{U}_b\sse clU$.

(b) If $x\in\widehat{U}_b$, then $b(x)\in X_b$, $U\in b(x)$, and $b(x)\in 0U$. Since $k\circ b=id_X$, we see that $x=k(b(x))$, thus $x\in k[0U \cap X_b]$. This shows $\widehat{U}_b\sse k[0U \cap X_b]$. The reverse containment is similar.

(c) follows from (b) above and Proposition~\ref{ex}(f).

(d) If $x\in X\backslash \widehat{U} _b$, then $U\notin b(x)$ and therefore $X\minus clU\in b(x)$. Thus $x\in(\widehat{X\backslash cl_XU})_b$ and $X\backslash \widehat{U} _b\sse (\widehat{X\backslash cl_XU})_b$. The reverse containment is identical.
\end{proof}
Recall that the semiregularization $X(s)$ of a space $X$ is the (Hausdorff) space with underlying set $X$ and topology generated by the basis of regular-open sets in $X$. 
\begin{corollary}
Let $X$ be a space, $U\in\tau(X)$, and $b:X\to EX$ be a section. Then $X_b=X(s)_b$.
\end{corollary}

\begin{proof}
Note that $b:X(s)\to EX(s)$ is also a section and $EX=E(X(s))$. By~\ref{hatu}(b), $\widehat{U}_b=k[0U\meet X_b]=k[0(int_Xcl_XU\meet X_b]=(\widehat{int_Xcl_XU})_b$. Thus, $X_b=X(s)_b$.
\end{proof}

\begin{proposition}\label{sectionTop}
Let $(X,\tau)$ be a space and $X_b$ a section of $EX$. Then $\{\widehat{U}_b: U \in \tau(X)\}$ is a clopen base for an extremally disconnected, Tychonoff topology $\sigma_b$ on $X$ such that $(X,\sigma_b)$ is homeomorphic to $X_b$.
\end{proposition}

\begin{proof}

The proof follows from ~\ref{hatu}(b) and the fact that $k|_{X_b}:X_b \rightarrow X$ is a bijection.
\end{proof}

For a space $X$ and a section $X_b$ of $EX$, define an operator $c_b:\scr{P}(X)\to\scr{P}(X)$ by
$$c_b(A)=\{x\in X: \widehat{U}_b\meet A\neq\es\textup{ for all }U\textup{ such that }x\in U\in\tau(X)\}.$$
We say that $A\sse X$ is $c_b$-\emph{closed} if $c_b(A)=A$.

Recall that for $A\sse X$, we define $aL(A,X)$ as the least infinite cardinal $\kappa$ such that if $\scr{V}$ is a cover of $A$ by sets open in $X$ then there exists $\scr{W}\in[\scr{V}]^{\leq\kappa}$ such that $A\sse\Un_{W\in\scr{W}}clW$. The \emph{almost Lindel\"of degree of $X$}, denoted by $aL(X)$, is $aL(X,X)$. The \emph{almost Lindel\"of degree of $X$ with respect to closed sets} is
$$aL_c(X)=\sup\{aL(C,X): C\textup{ is closed}\}+\aleph_0.$$
It is straightforward to see that $aL(X)\leq aL_c(X)\leq L(X)$ and that all three are identical if $X$ is regular.

\begin{definition}\label{alprime}
For a section $X_b$ of $EX$, we define the cardinal invariant $aL^b(X)$, by $aL^b(X)=\sup\{aL(C,X): C\textup{ is }c_b\textup{-closed}\}+\aleph_0$. As $aL^b(X)$ depends on the choice of section $X_b$, we define the unique cardinal invariant $aL^\prime(X)$ by
$$aL^\prime(X)=\min\{aL^b(X):X_b\textup{ is a section of }EX\}.$$
Let $X_{b^\prime}$ be any section witnessing that $aL^\prime(X)=aL^{b^\prime}(X)$ and define the operator $c:\scr{P}(X)\to\scr{P}(X)$ by $c=c_{b^\prime}$. We then refer to $aL^\prime(X)$ as the \emph{almost Lindel\"of degree of $X$ with respect to $c$-closed sets}. 
\end{definition}

\begin{notation}
For an open set $U$ of $X$, we let $\widehat{U}$ denote $\widehat{U}_{b^\prime}$. For $A\sse X$, we let $A^\prime = b^\prime[A]\sse EX$. For a point $x$ in a space $X$, let $\scr{U}_x$ represent the open ultrafilter $b^\prime(x)$. In general, throughout this study we will reserve the symbol ``$\scr{U}$'' to represent a convergent open ultrafilter. 
\end{notation}

The function $c$ defined above is the main operator in the closing-off argument used to prove our main result, Theorem~\ref{newbound}. We give properties of $c$ below.  For a subset $A\sse X$, we also give a characterization of $c(A)$ using $EX$ in Theorem~\ref{EXchar}.

\begin{proposition}\label{properties}
Let $X$ be a space, and $A, B\sse X$.
\begin{itemize}
\item[(a)] $A\sse c(A)$.
\item[(b)] if $A\sse B$ then $c(A)\sse c(B)$.
\item[(c)] $clA\sse c(A)\sse cl_\theta(A)$.
\item[(d)] if $U$ is open, then $clU=c(U)\sse c(\widehat{U})$. 
\item[(e)] if $X$ is regular then $clA=c(A)=cl_\theta(A)$.
\item[(f)] If $A$ is $c$-closed then $A$ is closed.
\end{itemize}
\end{proposition}
\begin{proof}
(a) If $x\in A$ and $U$ is an open set containing $x$, then $x\in\widehat{U}\meet A$ by Proposition~\ref{hatu}(a).

(b) If $x\in c(A)$ then $\es\neq \widehat{U}\meet A\sse\widehat{U}\meet B$ for all open sets $U$ containing $x$. This shows $x\in c(B)$.

(c) If $x\in clA$, then by Proposition~\ref{hatu}(a), $\es\neq U\meet A\sse\widehat{U}\meet A$ for all open sets $U$ containing $x$. Thus $x\in c(A)$. And if $x\in c(A)$, then $\es\neq\widehat{U}\meet A\sse clU\meet A$, also by Proposition~\ref{hatu}(a). Thus, $x\in cl_\theta(A)$.

(d) As $clU=cl_\theta U$ for an open set $U$, the equality follows from (b). The containment follows from (a) and Propositio~\ref{hatu}(a).

(e) As $clA=cl_\theta A$ for regular spaces, the result follows from (b).

(f) if $A$ is c-closed and $x\in X\minus A$, then there exists an open set $U$ containing $x$ such that $\es= \widehat{U}\meet A\supseteq U\meet A$. Thus $A$ is closed.
\end{proof}

The following example provides a space $X$ and a subset $A$ such that  $c(A) \ne cl_X(A)$. 

\begin{example}\label{c(A)notA}
Consider the Kat\v etov H-closed  extension $\kappa\omega$ of $\omega$ with the discrete topology (cf. Ch 7 in \cite{PW}).    Recall that $\kappa\omega(s) = \beta\omega$; that is, $\beta\omega$ is the underlying set of $\kappa\omega$. Also note that  $\kappa\omega\backslash \omega$ is discrete and closed.  By 9.11 in \cite{GJ}, the closed set $\beta\omega\backslash\omega$ contains a copy of $\beta\omega$. That is, there a countable discrete subspace $A$ of $\beta\omega\backslash \omega$ such that $\beta A = \beta\omega$, in particular, $cl_{\beta\omega}A = \beta\omega$.  Then $A$ is a closed subset of $\kappa\omega$.  Let $k: \beta\omega \rightarrow \kappa\omega$ denote the identity function. $\beta\omega$ is an extremally disconnected, Tychonoff space and the bijection $k$ is perfect, irreducible, and $\theta-$continuous.  By 6.7(a) in~\cite{PW},  $EX = \beta\omega$ is the absolute of $X = \kappa\omega$ with $k: EX \rightarrow X$ the absolute map. There is only one injective function $b:X \rightarrow EX$ such that $k\circ b = id_X$.  It follows that $c(A) = cl_{\beta\omega}(A)$.  By 9.3 in~\cite{GJ},  $c(A)$ has cardinality $2^{\mathfrak c}$; thus, $c(A) \ne cl_X(A) = A$.  This example also illustrates that there can be a marked size difference between $c(A)$ and $cl(A)$.
\qed
\end{example}

Observe that for the space $X = \kappa\omega$ in Example~\ref{c(A)notA},  we have $\sigma_b \subsetneq \tau$, where $\tau$ is the topology on $X$. However, by 5.1(d) in~\cite{pw78}, for a regular space $X$, we have that $\tau \subseteq \sigma_b$. Thus there is no universal containment relationship between $\tau$ and $\sigma_b$.

Let $X$ be a space, $U \in \tau(X)$, and $b:X \rightarrow EX$ be a section. By~\ref{hatu}(d), it follows that $\widehat{U}_b$ is also closed in $\sigma_b$. The next example shows that $\widehat{U}_b$ may not be c-closed.

\begin{example}
Let $\scr U$ and $\scr V$ be distinct free open ultrafilters on $\omega$, i.e., distinct  points in $\beta\omega\backslash\omega$.  Let $\alpha\omega$ denote the compactification of $\omega$ (discrete topology) where $\scr U$ and $\scr V$ in $\beta\omega$ are identified as the point $y$.  Let $X = \alpha\omega$.  Then $EX  = \beta\omega$ and $k:EX \rightarrow X$ is the identity function on $EX\backslash \{\scr {U,V}\}$ and $k(\scr U) = k(\scr V) = y$. Consider the section defined by the function $$b:X \rightarrow EX:x \mapsto \left \{\begin{array}{l@{\quad \quad}l} x & x \ne y \\  \vspace{0mm} \scr U & x = y.\end{array} \right.$$ 
For $A \in [\omega]^{\omega}$, let 
$$o_{\alpha}A = A \cup \{\scr W \in \beta\omega\backslash \{\scr {U,V}\}: A \in \scr W\} \cup \left \{\begin{array}{l@{\quad \quad}l} \varnothing & A \not\in  \scr U \cap \scr V \\  \vspace{0mm} \{y\} & A \in  \scr U \cap \scr V\end {array}\right. \}.$$  
Then $\{o_{\alpha}A:A \in [\omega]^{\omega}\} \cup \{ \{n\}: n \in \omega \}$ is a base for $\tau(\alpha\omega)$. In particular, a neighborhood base for $y$ is $\{o_{\alpha}W:W \in {\scr U} \cap {\scr V}\}$.  Also, note that for $A \in [\omega]^{\omega}$, $$\widehat{o_{\alpha}A} =o_{\alpha}A \cup \left \{\begin{array}{l@{\quad \quad}l} \varnothing & A \not\in  \scr U \\  \vspace{0mm} \{y\} & A \in  \scr U \end {array}\right. .$$  Let $T \in \scr V\backslash \scr U$.  Then $y \not\in \widehat{o_{\alpha}T}$.  For $W \in {\scr U} \cap {\scr V}$,    $y \in o_{\alpha}W $ and $\widehat{o_{\alpha}W} \cap \widehat{o_{\alpha}T} \supseteq W \cap T \ne \varnothing$.  That is, $c(\widehat{o_{\alpha}T}) = \widehat{o_{\alpha}T} \cup \{y\} \ne \widehat{o_{\alpha}T}$. \qed
\end{example}

In view of Proposition~\ref{properties}(f) and the fact that any space is $c$-closed in itself, we have the following:

\begin{corollary}\label{anotherProp} 
For any space $X$, $aL(X)\leq aL^\prime(X)\leq aL_c(X)\leq L(X)$.
\end{corollary}

For a space $X$, we define a cardinal invariant $t_c(X)$ related to the tightness $t(X)$. While $t_c(X)$ and $t(X)$ appear to be incomparable, Proposition~\ref{tcx} shows that $t_c(X)$ is bounded above by the character $\chi(X)$. 

\begin{definition}
For a space $X$, the $c$-\emph{tightness} of X, $t_c(X)$, is defined as the least cardinal $\kappa$ such that if $x\in c(A)$ for some $x\in X$ and $A\sse X$, then there exists $B\in[A]^{\leq\kappa}$ such that $x\in c(B)$.
\end{definition}

Note that $t(\kappa\omega) = \aleph_0$ and $t_c(\kappa\omega) = t(\beta\omega) = \frak{c}$. This shows that $t(\kappa\omega) $ and $t_c(\kappa\omega)$ are not equal.

\begin{proposition}\label{tcx}
For any space $X$, $t_c(X)\leq\chi(X)$. If $X$ is regular then $t_c(X)=t(X)$.
\end{proposition}

\begin{proof}
To show $t_c(X)\leq\chi(X)$, let $\kappa=\chi(X)$ and let $x\in c(A)$. Let $\scr{N}$ be an open neighborhood base at $x$ such that $|\scr{N}|=\kappa$. For all $N\in\scr{N}$ there exists $a_N\in\widehat{N}\meet A$. Let $D=\{a_N:N\in\scr{N}\}\in[A]^{\leq\kappa}$. We show $x\in c(D)$. Let $U$ be an open set containing $x$. There exists $N\in\scr{N}$ such that $x\in N\sse U$. As $a_N\in\widehat{N}\meet A$, we have $N\in\scr{U}_{a_N}$. As $\scr{U}_{a_N}$ is an open filter, we have that $U\in\scr{U}_{a_N}$ and $a_N\in\widehat{U}\meet D$. This shows $x\in c(D)$ and $t_c(X)\leq\kappa$. If $X$ is regular then $t_c(X)=t(X)$ follows from Proposition~\ref{properties}(e).
\end{proof}

In Theorem~\ref{EXchar}, we give a characterization of $c(A)$ for a subset $A\sse X$ in terms of the absolute $EX$. This is one of several results below that describe how $c(A)$ relates to the broader framework of $EX$.

Let $\scr K =\{k^{\leftarrow}(p): p \in X\}$, where $k:EX\to X$ is as in~\ref{ex}(a). For $A \subseteq EX$, define $cl_{\scr K}A = A \cup \bigcup\{K \in \scr K:$ if $K \subseteq 0U$ for  $U \in \tau(X), 0U \cap A \ne \varnothing \}$. 

\begin{lemma}\label{closures}
For $A \subseteq EX$, $cl_Xk[A] \subseteq k[cl_{EX}A] \subseteq k[cl_{\scr K}A] \subseteq cl_{\theta}k[A]$.
\end{lemma}

\begin{proof} As $k$ is a closed function, we immediately have that $cl_Xk[A] \subseteq k[cl_{EX}A]$. To show $k[cl_{EX}A] \subseteq k[cl_{\scr K}A]$, it suffices to show that  $cl_{EX}A \subseteq cl_{\scr K}A$.  Let $\scr U \in cl_{EX}A$ and $K = k^{\leftarrow}(k(\scr U) \subseteq 0U$ for some  $U \in \tau(X)$.  Then $\scr U \in 0U$ and $0U \cap A \ne \varnothing$ as $\scr U \in cl_{EX}A$.  Thus, $\scr U \in K \subseteq cl_{\scr K}A$.
 \newline Now, let $p \in k[cl_{\scr K}A]$ and $p \in U  \in \tau(X)$.  Then, by Proposition~\ref{ex}(d), $ k^{\leftarrow}(p) \subseteq 0U$.  So, $0U \cap A \ne \varnothing$ and $\varnothing \ne k[0U \cap A] \subseteq k[0U] \cap k[A] \subseteq cl_X(U) \cap k[A]$. Therefore, $x \in cl_{\theta}k[A]$. 
\end{proof}

\begin{lemma}\label{clKchar}
For $A\sse EX$, 
$$cl_{\scr{K}}A=\Un\{K\in\scr{K}:K\meet cl_\theta A\neq\es\}=\Un k^{\leftarrow}[k[cl_\theta A]]=k^{\leftarrow}[k[cl_{EX}A]].$$
\end{lemma}

\begin{proof}
Suppose $K\meet cl_\theta A=\es$ for $K\in\scr{K}$. Then for each $\scr{U}\in K$, there is $U_\scr{U}$ such that $0(U_\scr{U})\meet A=\es$ and $\{0(U_\scr{U}):\scr{U}\in K\}$ is an open cover of the compact set $K$. There exists $\scr{U}_1,\ldots,\scr{U}_n\in K$ such that 
$$K\sse\Un_{i=1}^n 0(U_{\scr{U}_i})=0(\Un_{i=1}^nU_{\scr{U}_i}),$$
by~\ref{ex}(f). Let $U=\Un_{i=1}^nU_{\scr{U}_i}$. As $K \subseteq 0(U)$ and $0(U) \cap A = \varnothing$, we see that $K \cap cl_{\scr K}A = \varnothing$. Conversely, suppose $\scr{U}\in K\meet cl_\theta A$ and $K\sse 0(U)$. Then $\scr{U}\in 0(U)$ and $0(U)\meet A\neq\es$. Thus $K\sse cl_\scr{K}A$. This shows the first equality.

To show $\Un\{K\in\scr{K}:K\meet cl_\theta A\neq\es\}=\Un k^{\leftarrow}[k[cl_\theta A]]$, it suffices to note that if $\scr{U}\in cl_\theta A$, then $k^{\leftarrow}[k(\scr{U})]\meet cl_\theta A\neq\es$ and $k^\leftarrow[k(\scr{U})]\in\scr{K}$. The equality $\Un k^{\leftarrow}[k[cl_\theta A]]=k^{\leftarrow}[k[cl_{EX}A]]$ follows as $EX$ is Tychonoff.
\end{proof}

\begin{theorem}\label{EXchar}
Let $X$ be a space and $A \subseteq X$.  Then $c(A) = k[cl_{\scr K}A']=k[cl_{EX}A^\prime]$.
\end{theorem}

\begin{proof} We first show the first equality. Clearly, $A \subseteq c(A) \cap k[cl_{\scr K}A']$.  Let $p \in c(A)\backslash A$ and $k^{\leftarrow}(p) \subseteq 0U$ where $U \in \tau(X)$.  By~\ref{ex}(d), $p \in int(cl(U))$.  There is $a \in A$ such that $int(cl(U)) \in\scr{U}_a$.  By~\ref{ex}(b), $U \in\scr{U}_a \subseteq k^{\leftarrow}(a)$.  Thus, $\scr{U}_a \in 0(int(cl(U))) = 0(U)$.  That is, $\scr{U}_a \in 0(U) \cap A'$.  This shows that $k^{\leftarrow}(p) \subseteq cl_{\scr K}A'$ and $p\in k[cl_{\scr K}A']$.  Conversely suppose $p\in k[cl_{\scr K}A']\backslash A$.  Let $p \in U \in \tau(X)$.  Then $k^{\leftarrow}(p) \subseteq 0(U)$ and $0(U) \cap A' \ne \varnothing$.  There is $a \in A$ such that $\scr{U}_a \in 0(U)$. Thus, $U \in\scr{U}_a$ and $p \in c(A)$. This shows that $c(A) = k[cl_{\scr K}A']$.

To show the second equality, note that by Lemma~\ref{clKchar}, we have
$$c(A)=k[cl_{\scr K}A']=k[k^{\leftarrow}[k[cl_{EX}A^\prime]]]= k[cl_{EX}A^\prime].$$
\end{proof}

As the map $k:EX\to X$ is always a closed map, we have the following corollary to Theorem~\ref{EXchar}.

\begin{corollary}
For any space $X$ and every $A\sse X$, $c(A)$ is closed subset of $X$.
\end{corollary}

By Theorem~\ref{EXchar} and Proposition~\ref{properties}(c), we also have the following corollary. We see that Corollary~\ref{ccorollary} is stronger than Proposition~\ref{properties}(c) and demonstrates how $c(A)$ sits between $cl_XA$ and $cl_\theta A$ in terms of the absolute $EX$.

\begin{corollary}\label{ccorollary}
Let $X$ be a space and $A \subseteq X$. Then $cl_XA \subseteq k[cl_{EX}A']= c(A) \subseteq cl_{\theta}A$.
\end{corollary}

Our next corollary to Theorem~\ref{EXchar} demonstrates that the $c-$closure of a subset of an H-closed space is both Kat\v etov and an H-set. This result should be compared with Lemma~\ref{lemma2} which gives different conditions under which a subset of an H-closed space is an H-set and the result from \cite{JJ} that the $\theta-$closure of a subset of an H-closed space is an H-set.

\begin{corollary}
If $X$ is H-closed and $A\sse X$, then $c(A)$ is Kat\v etov and an H-set.
\end{corollary}

\begin{proof}
As $c(A)=k[cl_{EX}A^\prime]$ by Theorem~\ref{EXchar}  and since the absolute $EX$ is compact when $X$ is H-closed (\cite{PW} 6.9(b)(1)), we have that $cl_{EX}A^\prime$ is compact and  $c(A)$ is Kat\v etov and  an H-set by~\ref{ex}(e).
\end{proof}

\vskip 3mm
Another consequence of Theorem~\ref{EXchar} are the following properties of the c-closure operator and a new characterization of H-closed spaces.

\vskip 3mm 

\begin{proposition}\label{c-prop} Let $X$ be a space and $A, B$ subsets of $X$.
\newline (a) If $A \subseteq B$, then $c(A) \subseteq c(B)$.
\newline (b) $c(A \cap B) \subseteq c(A) \cap c(B)$.
\newline (c) $c(A \cup B) = c(A) \cup c(B)$.
\end{proposition}

\begin{proof} (a) is immediate from \ref{EXchar}, and (b) follows from (a).  For (c), note that $c(A \cup B)  = k[cl_{EX}(A \cup B)'] = k[cl_{EX}(A' \cup B')' = k[cl_{EX}(A') \cup cl_{EX}(B')] = k[cl_{EX}(A')] \cup k[cl_{EX}(B')] = c(A) \cup c(B)$. 
\end{proof}
\vskip 3mm
\begin{definition} Let $\scr F$ be a filter base on a space $X$.  We define the c-adherence of $\scr F$, denoted as $a_c(\scr F)$, as $\cap\{c(F):F 
\in \scr F\}$.  By \ref{properties}(c), it follows that $a(\scr F) \subseteq a_c(\scr F) \subseteq a_{\theta}(\scr F)$.
\end{definition}

We will use the concept of c-adherence to obtain a new characterization of H-closed spaces in the next result. 

\begin{theorem}\label{H-cldChar} Let $X$ be a space.  Then  $X$ is H-closed iff
for every filter base $\scr F$ on $X$, $a_c(\scr F) \ne \varnothing$.
\end{theorem}

\begin{proof}  Let $\mathcal F$ be a filter base on $X$. To show $X$ is H-closed, it suffices to show that $a_{\theta}(\mathcal F) \ne \varnothing$.  But $c(F) \subseteq cl_{\theta}(F)$ for each $F \in \mathcal F$ and $a_c(\mathcal F) \ne  \varnothing$. Thus, $a_{\theta}(\mathcal F) \ne \varnothing$.  Conversely suppose $X$ is H-closed.  Let $\mathcal F$ be a filter base on $X$. Then $\{cl_{EX}F':F \in \mathcal F \}$ is a filter base of compact subsets on $EX$.  Thus, there is $p \in \cap \{cl_{EX}F':F \in \mathcal F \}$.  It follows that $k(p) \in a_c(\mathcal F)$.
\end{proof}

In the next section we will develop further connections between the H-closed property, the operator $c$, and the set $\widehat{U}$ for an open set $U\sse X$.

\section{H-closed spaces.}

For a space $X$ we define $L^\prime(X)$, a cardinal invariant related to $aL^\prime(X)$, and show in Proposition~\ref{Lhered} that it is hereditary on $c$-closed subsets. 
The filter characterization of H-closed spaces used in Theorem~\ref{H-cldChar} using $c$-adherence of a filter in conjunction with a variation of a method used by Hodel ~\cite{Hodel2006} for nets provides a direct path for proving that the cardinality of an H-closed space $X$ is bounded by $2^{\psi_c(X)}$.

\begin{definition}\label{lprime}
For a subset $A\sse X$, define $L^\prime(A,X)$ as the least cardinal $\kappa$ such that for every cover $\scr{V}$ of $A$ by sets open in $X$ there exists $\scr{W}\in[V]^{\leq\kappa}$ such that $A\sse\Un_{W\in\scr{W}}\widehat{W}$. Set $L^\prime(X)=L^\prime(X,X)$. 
\end{definition}

\begin{proposition}\label{Lhered}
Let $X$ be a space. If $A\sse X$ is $c$-closed, then $L^\prime(A,X)\leq L^\prime(X)$.
\end{proposition}

\begin{proof}
Let $\kappa=L^\prime(X)$ and let $\scr{V}$ be a cover of $A$ by sets open in $X$. As $A$ is $c$-closed, for all $x\in X\minus A$, there exists an open set $W_x$ containing $x$ such that $a\notin\widehat{W_x}$ for all $a\in A$. Let $\scr{W}=\{W_x:x\in X\minus A\}$. Then $\scr{W}\un\scr{V}$ is an open cover of $X$. As $L^\prime(X)=\kappa$, there exists $\scr{W}^\prime\in[\scr{W}]^{\leq\kappa}$ and $\scr{V}^\prime\in[\scr{V}]^{\leq\kappa}$ such that
$$X=\Un_{W\in\scr{W}^\prime}\widehat{W}\un\Un_{V\in\scr{V}^\prime}\widehat{V}.$$
Suppose there exists $a\in A\meet\Un_{W\in\scr{W}^\prime}\widehat{W}$. Then there exists $W\in\scr{W}^\prime$ such that $a\in\widehat{W}$, a contradiction.  Thus $A\meet\Un_{W\in\scr{W}^\prime}\widehat{W}=\es$ and $A\sse\Un_{V\in\scr{V}^\prime}\widehat{V}$. This shows $L^\prime(A,X)\leq\kappa$.
\end{proof}

\begin{corollary}\label{aLL}
For any space $X$, $aL^\prime(X)\leq L^\prime(X)\leq L(X)$.
\end{corollary}

\begin{proof}
To show the first inequality, let $\kappa=L^\prime(X)$, let $A$ be a $c$-closed subset of $X$, and let $\scr{V}$ be a cover of $A$ by sets open in $X$. As $L^\prime(A,X)\leq\kappa$ by Propostion~\ref{Lhered}, there exists $\scr{W}\in[\scr{V}]^{\leq\kappa}$ such that $A\sse\Un_{W\in\scr{W}}\widehat{W}$. By Proposition~\ref{hatu}(a), we see that $A\sse\Un_{W\in\scr{W}}\widehat{W}\sse\Un_{W\in\scr{W}}clW$. This shows $aL^\prime(X)\leq\kappa$. To see that $L^\prime(X)\leq L(X)$, just observe again by Proposition~\ref{hatu}a that $V\sse\widehat{V}$ for every member $V$ of an open cover of $X$.
\end{proof}
In addition to Theorem~\ref{H-cldChar}, we obtain another new characterization of H-closed spaces.

\begin{theorem}\label{HclosedChar}
A space $X$ is H-closed if and only if for every open cover $\scr{V}$ of $X$ there exists a finite family $\scr{W}\sse\scr{V}$ such that $X=\Un_{W\in\scr{W}}\widehat{W}$.
\end{theorem}

\begin{proof}
Let $\scr{V}$ be an open cover of $X$ and suppose there exists a finite family $\scr{W}\sse\scr{V}$ such that $X=\Un_{W\in\scr{W}}\widehat{W}$. Then, by Proposition~\ref{hatu}(a), $X=\Un_{W\in\scr{W}}clW$, showing $X$ H-closed.

Suppose now that $X$ is H-closed and let $\scr{V}$ be an open cover of $X$. As $X$ is H-closed, there is a finite family $\scr{W}\in\scr{V}$ such that $X=\Un_{W\in\scr{W}}clW$. Suppose by way of contradiction that there exists $x\in X\minus(\Un_{W\in\scr{W}}\widehat{W})$. Then, by Proposition~\ref{hatu}(d), 
$$x\in X\minus(\Un_{W\in\scr{W}}\widehat{W})=\Meet_{W\in\scr{W}}(X\minus\widehat{W})=\Meet_{W\in\scr{W}}(\widehat{X\minus clW}).$$
Then, $X\minus clW$ is a member of the open ultrafilter $\scr{U}_x$ for all $W\in\scr{W}$. It follows by the finite intersection property that
$$\es=\Meet_{W\in\scr{W}}(X\minus clW)\in\scr{U}_x.$$
As this is a contradiction, we see $X=\Un_{W\in\scr{W}}\widehat{W}$.
\end{proof}
We have the following immediate corollary of Theorem~\ref{HclosedChar}.
\begin{corollary}\label{hclosed}
If $X$ is H-closed then $L^\prime(X)=\aleph_0$.
\end{corollary}

For the space $X=\kappa\omega$ in Example~\ref{c(A)notA}, we note by~\ref{hclosed} that $L^\prime(X)=\aleph_0$. Yet $L(X)=2^\mathfrak{c}$. Furthermore, since $X_b=X(s)_b=\beta\omega$ for the section $X_b$ in that example, it follows that $L(X(s))=\aleph_0$.

\vskip 1mm

We now present an example of an H-closed space $X$ and a subset $A$ such that $cl_X(A) \ne c(A) \ne cl_{\theta}(A)$ showing that \ref{properties}(c) is the best general result.

\vskip 3mm

\begin{example}\label{Urys}
We use Urysohn's space  $\Bbb U$ defined in 1925 to show that the converse of Proposition~\ref{properties}(c)  is not true in the setting of H-closed spaces.  Let $\Bbb Z$
denote the set of all integers with the discrete topology and $\Bbb N$
denote the subspace of positive integers.
  For the set $\Bbb U  = \Bbb N \times
\Bbb Z \cup \{\pm \infty\}$, a subset $U\subseteq \Bbb U$  is defined to be
open  if $+\infty\in U$
(resp. $-\infty\in U$) implies for some
$k\in {\Bbb N}$, $\{(n,m):n\geq k, m\in {\Bbb N} \}\subseteq U$ (resp.
$\{({n},{-m}):n\geq k, m\in
{\Bbb N}\}\subseteq U)$ and if $(n,0) \in U$ implies for some $k\in {\Bbb
N}$, $\{({n},{\pm m}):m\geq k
\}\subseteq U)$. The space $\Bbb U$ is first countable,  minimal
Hausdorff (H-closed and semiregular) but is not compact as
$A = \{(n,0): n
\in {\Bbb N}\} $ is an infinite, closed discrete subset.  Let $k: E{\Bbb U} \rightarrow {\Bbb U}$  be the absolute map from the absolute  $E\Bbb U$ to $\Bbb U$. Let $\mathcal U \in k^{\leftarrow}(\infty)$ such that $\Bbb N \times \{2\} \in \mathcal U$; thus, $\mathcal U \rightarrow \infty$.  Let    $\mathcal V \in k^{\leftarrow}(-\infty)$ such that $\Bbb N \times \{-2\} \in \scr V$; thus, $\mathcal V \rightarrow -\infty$. For $n \in \Bbb N$, let $\mathcal U_n \in k^{\leftarrow}((n,0))$ such that $\{n\}\times \Bbb N \in \mathcal U_n$; thus, $\mathcal U_n \rightarrow (n,0)$.  
Define $b:\Bbb U \rightarrow E\Bbb U$ by $b(\infty) = \mathcal U, b(-\infty) = \mathcal V, b((n,0)) = {\mathcal U}_n,$ and for $(n,m) \in \Bbb N \times \Bbb Z \backslash \Bbb N \times\{0\}$, $b(n,m) = \{U \in \tau(\Bbb U): (n,m) \in U\}$.  It follows that $cl_{\Bbb U}(A) = A$ and $cl_{\theta}(A) = A \cup \{\pm \infty\}$. By \ref{properties}(c),   it follows $A \sse c(A) \sse A \cup \{\pm\infty\}$.  To show that $\infty \in c(A)$, for $n \in \Bbb N$, let $T_n = (\Bbb N \backslash \{1,2,\cdots,n\})\times \Bbb N$.  A basic open set containing $\infty$ is  $T_n \cup \{\infty\}$.  As $\{n+1\}\times \Bbb N \in \scr U_{n+1} = b(n+1,0)$, $T_n \in \scr U_{n+1}$ and $\hat{T_n} \cap A \ne \varnothing$.  A similar argument shows that $-\infty \not\in c(A)$.
Thus, $c(A) = A \cup \{\infty\}$ and this shows that $cl_X(A) \ne c(A) \ne cl_{\theta}(A)$.  Also, note that both 
 $c(A)$ and $cl_{\theta}(A)$ are H-sets.  \qed
\end{example}

\begin{definition}\label{kHcl}
Let $X$ be a space, $\ka$ an infinite cardinal, and $A \sse X$.
$A$ is $\kappa$-\emph{H-closed} if for each open (in $X$) cover $\mathcal C$ of $A$ such that $|{\mathcal C}| \leq \ka$, there is a finite subfamily $ {\mathcal D}\subseteq {\mathcal C}$ such that $A \subseteq \bigcup_{\mathcal D}cl_X(U \cap A)$).  
\end{definition}

We note that in particular, a Hausdorff space $X$ is $\omega$-H-closed iff $X$ is feebly compact. 

We prove the following lemmas. The key lemma is Lemma~\ref{lemma2}, which is of interest on its own.  

\begin{lemma}\label{lemma1}
Let $X$ be a space, $\ka$ an infinite cardinal, and $A \subseteq X$.  If for each filter base ${\mathcal F} \in [[A]^{\leq \ka}]^{\leq \ka}$, $a_c({\scr F}) \cap A \ne \varnothing$,  then $A$ is a  $\ka$-H-closed.
\end{lemma}

\begin{proof} Let $\scr C$ be an open cover of $A$ by sets open in $X$ and suppose that $\scr C$ is closed under finite unions and suppose $|{\scr C}| \leq \kappa$.  For each $V \in \scr C$, assume there is $p_V \in A\backslash cl(V\cap A)$.  Let $B_V = \{p_U: V \sse U \in \scr C\}$.  For $T, S \in \scr C$, $B_T \cap B_S \supseteq B_{T\cup S}$.  Then $\scr F = \{B_V: V \in \scr C\}$ is a filter base on $A$.  Thus, there is a point $p \in a_c(\scr F) \cap A$.  There is $T \in \scr C$ such that $p \in T$.  Now, $B_T \sse A\backslash cl(T\cap A) \sse X\backslash cl(T\cap A)$.  By Propositions~\ref{properties}(d) and \ref{c-prop}(a), using that $X\backslash cl(T\cap A)$ is open, we have that $ p \in c(B_T) \sse c(X\backslash cl(T)) = cl(X\backslash cl(T)) \sse X\backslash T$, a contradiction as $p \in T$.\end{proof}

The small filter base method presented in the above lemma stands in contrast to the open ultrafilter techniques frequently used in H-closed settings. 
An immediate consequence is this corollary.

\begin{corollary} Let $X$ be a space and $A \sse X$.  If for every filter base $\scr F$ on $A$, $a_c(\scr F) \cap A  \ne \varnothing$, then the subspace  $A$ is H-closed.
\end{corollary}

\begin{lemma}\label{lemma2}
Let $X$ be an H-closed space, $\ka$ an infinite cardinal, $A\sse X$, and $\psi_c(X) \leq \kappa$.  If  for every filter base $\scr F$ on $[[A^{\leq \kappa}]^{\leq \kappa}]$, $a_c(\scr F) \cap A  \ne \varnothing$.  Then $A$ is an H-set.
\end{lemma}
\begin{proof}
Let $\scr G$ be an open filter that meets $A$.  We can assume that $\scr G$ is maximal with respect to meeting $A$.  As $X$ is H-closed, there is $p \in a(\scr G)$.  The goal is to show that $p \in A$.  Assume that $p \not\in A$. There is a family $\scr V = \{V_{\alpha}: \alpha < \kappa \}$ of open neighborhoods of $p$ such that $\cap_{\kappa}cl(V_{\alpha}) = \{p\}$. For each $V_{\al}$, as $p \not\in cl(X\backslash cl(V_{\al}))$, $X\backslash cl(V_{\al}) \not\in \scr G$.  There is some $G_{\al} \in \scr G$ such that $(X\backslash cl(V_{\al})) \cap G_{\al} \cap A = \varnothing$.  Thus, $G_{\al} \cap A \sse cl(V_{\al})$ and $cl(G_{\al} \cap A) \sse cl(V_{\al})$.  Assume, by way of contradiction, that $A \cap \bigcap cl(G_{\al} \cap A) = \varnothing.$  Thus, $\{X\backslash cl(G_{\al} \cap A): \al < \ka\}$ is open cover of $A$.  As $A$ is $\kappa$-H-closed by Lemma~\ref{lemma1}, there is a finite set $F \in [\ka]^{<\omega}$ such that $A \subseteq \bigcup_Fcl((X\backslash cl(G_{\al} \cap A))\cap A)$.  There is a $G \in \scr G$ such that $G \sse \bigcap_FG_{\al}$.  For $\al \in F$, $G \cap A \sse G_{\al}\cap A$ implying that $X\backslash cl(G_{\al}\cap A) \sse X\backslash cl(G \cap A )$ and  $(X\backslash cl(G_{\al}\cap A))\cap A \sse (X\backslash cl(G \cap A ))\cap A$.
Thus $A \sse cl((X\backslash cl(G \cap A))\cap A) = cl(A\backslash cl_A(G\cap A))$ implying that $A \sse  cl_A(A\backslash cl_A(G\cap A)) = A\backslash int_Acl_A(G\cap A) $,  a contradiction as $int_Acl_A(G\cap A)\cap A \ne \varnothing$.
\end{proof}

\begin{theorem}\label{kappafilter}
Let $X$ be H-closed,  $\ka$ an infinite cardinal, and $\psi_c(X) \leq \ka$.  Then $|X|\leq 2^{\ka}$.
\end{theorem}

\begin{proof}
For each $x \in X$, let $\{V(\al, x) : \al \in \ka\}$ be a family of open sets containing $x$ such that $\bigcap_{\ka}cl(V(\al,x)) = \{x\}$. Let $L:{\mathcal P}(X) \rightarrow X$ be a choice function.  Using transfinite induction, we will construct a sequence $\{H_{\alpha}: 0\leq \alpha<\kappa^+\}$ of subsets of $X$ such that for $0\leq \alpha<\kappa^+$: 

(a) $H_0 = \{L(\varnothing)\}$; 

(b) if $H_{\beta}$ is defined for $\beta < \alpha$, define $H_{\al}$ as follows:
$$f((\bigcup_{\beta <\alpha}H_{\beta})
\cup 
\{L(X \backslash \bigcup_{x \in A} cl(V(x,g(x))) : A \in [\bigcup_{\beta <\alpha}H_{\beta}]^{<\omega}, g :A \rightarrow \ka\}).$$
Note that $|H_{\alpha} | \leq 2^{\kappa}$ for $0 \leq \alpha < \kappa^+$.
Let $H = \bigcup\{H_{\al}: \al  < \ka^+\}$. It follows that $|H| \leq 2^{\ka}$ and $f(H) \subseteq H$. Thus, $H = f(H)$ and  if ${\mathcal F} \in [[H]^{\ka}]^{\ka}$, $a^H_{\mathcal F}({\mathcal F}) \ne \varnothing$. By Lemmas~\ref{lemma1} and~\ref{lemma2}, $H$ is an H-set.

To show that $H = X$, assume that  $q \notin H$. Since $\psi_c(X)\leq \ka$, for each $x \in H$, there is $\al_x < \ka$ such that $q \not\in clV(\al_x,x)$.  Using that $H$ is H-set, there a finite subset $A \in [H]^{<\omega}$ such that $H \subseteq \bigcup_Acl(V(\al_x,x)) \subseteq X\backslash \{q\}$.
Now choose $\al <\ka^+$ such that $A\in [\bigcup_{\beta <\al}H_{\beta}]^{<\omega}$. By (b), $L(X\backslash \bigcup_{x \in A} cl(V(\al_x,x)) \in H_{\al}$ and it follows that $H_{\alpha}\backslash \bigcup_{x \in A} cl(V(\al_x,x)) \ne \varnothing$.  This is a contradiction as $H_{\alpha} \subseteq H \subseteq \bigcup_{x \in A}cl(V(\al_x,x))$.
\end{proof}

It follows from Theorem~\ref{kappafilter} that the cardinality of an H-closed space is at most $2^{\chi(X)}$. The Dow-Porter result given in Corollary~\ref{dowporter} now follows, using a proof similar to the proof of that corollary. We see then two very different proofs of this result, one using open ultrafilters (which generalizes to a result for all Hausdorff spaces, Theorem~\ref{newbound}) and the other using $\kappa$-nets~\cite{Hodel2006} which can be reframed in terms of $\kappa$-filters as in Theorem~\ref{kappafilter}. We note that in~\cite{por93}, Porter used a different type of open ultrafilter approach.

We present several examples. 

\begin{example} 
This example demonstrates that the converse of Lemma~\ref{lemma1} is false, i.e.,  an $\omega$-H-closed space $X$ with a filter base ${\scr F} \in [[X]^{\leq \omega}]^{\leq \omega}$  such $a_{c}({\scr F}) = \varnothing$.  The space $X$ is Tychonoff; so, $a_{c}({\scr F}) = a_{\theta}({\scr F}) = a({\scr F})$.

Consider the partition $\{A_n:n \in \omega\}$  of  $\omega$  where each $A_n$ is infinite. Pick one point, say $a_n \in cl_{\beta\omega}A_n\backslash A_n$.  Let $B = cl_{\beta\omega}\{a_n:n \in \omega\}\backslash \{a_n:n \in \omega\}$.  
We will show that the subspace $X = \beta\omega\backslash B$ is $\omega$-H-closed but has a filter base ${\scr F} \in [[X]^{\leq \omega}]^{\leq \omega}$  such $a_{c}({\scr F}) = \varnothing$.
For $n \in \omega$, let $F_n = \{a_m: m \geq n\}$; thus, ${\scr F} = \{F_n: n \in \omega\} \in [[X]^{\leq \omega}]^{\leq \omega}$.

Let $x \in X$. Then, in $\beta\omega$, $B \cup \{a_n:n \in \omega\}\backslash \{x\}$ is compact and there are disjoint open sets $U,V$ in $\beta\omega$ such that $B \cup \{a_n:n \in \omega\}\backslash \{x\} \subseteq U$ and $x \in V$.  $U\backslash B $ and $V\backslash B$ are disjoint open sets in $X$ such that $ \{a_n:n \in \omega\}\backslash \{x\} \subseteq U\backslash B$ and $x \in V\backslash B$.
If $x = a_m$ for some $m \in \omega$, then $F_{m+1}\cap cl_(X(V\backslash B) = \varnothing$ and $x \not\in a_{c}({\scr F})$.
If $x \not\in  \{a_n:n \in \omega\}$, then $F_0 \cap cl_X(V\backslash B)= \varnothing$ and $x \not\in a_{c}({\scr F})$.
So in both cases, $x \not\in a_{c}({\scr F})$ and $a_{c}({\scr F}) = \varnothing.$ 

To show that $X$ is $\omega$-H-closed (= feebly compact), it suffices to show that the Tychonoff space $X$ is pseudocompact by 1.10(d)(2) in~\cite{PW}. It suffices, by 1Q(6) in~\cite{PW} to show that every infinite subset of $\omega$ is not closed in $X$.  Let $C = \{b_n:n \in \omega\}$ be infinite subset of $\omega$.  As $cl_{\beta\omega}A_n \cap B = \varnothing$ for each $n \in \omega$ and $cl_{\beta\omega}(A_n \cap C)\backslash \omega \ne \varnothing$ whenever $A_n \cap C$ is infinite, it follows that $A_n \cap C$ is finite for each $n \in \omega$.  Thus, by 4B(6) in~\cite{GJ}, $\{a_n:n \in \omega\}$ and $C$ are contained in disjoint cozero-sets (in an extremely disconnected space) and hence $B \cap cl_{\beta\omega}C = \varnothing$. It follows that $cl_{\beta\omega}C\backslash C \subseteq X$.\qed

\end{example}

\begin{example} 
The $\psi$ space $X$ is an example of a first countable, Tychonoff, pseudocompact space with a filter base ${\scr F} \in [[X]^{\leq \frak c}]^{\leq \frak c}$  such that $a_{c}({\scr F}) = \varnothing$.
Let $X = \omega \cup {\scr M}$ where ${\scr M}$ is a maximal family of almost disjoint infinite subsets of $\omega$ and $U \subseteq X$ is open if $A \in {\scr M} \cap U$ implies there is a $F \in \omega^{<\omega}$ such that $A\backslash F \subseteq U$. It is well-known that $X$ is first countable, locally compact, Tychonoff, pseudocompact space that is not countably compact and $|{\scr M}| = \frak c$.  

For $B \in [{\scr M}]^{<\omega}$, let $F_B = {\scr M}\backslash B$  and ${\scr F} = \{F_B: B \in [\scr{M}]^{<\omega}\}$.   We will show that  such $a_{c}({\scr F}) = \varnothing$.  Let $x \in X$.  If $x \in \omega$, then $\{x\}$ is a clopen set disjoint from  ${\scr M} \in {\scr F}$. If $x = A \in {\scr M}$, then if ${\scr M}\backslash \{A\} \in {\scr F}$, then as  $cl_X(\{A\} \cup A) = \{A\} \cup A $, $cl_X(\{A\} \cup A) \cap ({\scr M}\backslash \{A\}) = \varnothing$.
Thus, $a_{c}({\scr F}) = \varnothing$. \qed
\end{example}

For the space $\Bbb U$ constructed in Example~\ref{Urys}, the subspace $\{(n,0): n \in \Bbb N\} \cup \{\infty\}$ is an $\omega$-H-set but not $\omega$-H-closed.

\section{A new cardinality bound for Hausdorff spaces.}

\begin{proposition}\label{intersection}
If $X$ is Hausdorff and $\psi_c(X)\leq\kappa$, then for all $x\in X$ there exists a family $\scr{V}$ of open sets such that $|\scr{V}|\leq\kappa$ and 
$$\{x\}=\Meet\scr{V}=\Meet_{V\in\scr{V}}clV=\Meet_{V\in\scr{V}}c(\widehat{V}).$$
\end{proposition}

\begin{proof}
Fix $x\in X$. As $\psi_c(X)\leq\kappa$, there exists a family $\scr{V}$ of open sets such that $\{x\}=\Meet\scr{V}=\Meet_{V\in\scr{V}}clV$ and $|\scr{V}|\leq\kappa$. Suppose $y\neq x$. There exists $V\in\scr{V}$ such that $y\in X\minus clV$. Let $W=X\minus clV$ and suppose $y\in c(\widehat{V})$. Then, 
$$\es\neq\widehat{W}\meet\widehat{V}=\widehat{W\meet V}=\widehat{\es}=\es,$$
a contradiction. Thus $y\notin c(\widehat{V})$ and $\{x\}=\Meet_{V\in\scr{V}}c(\widehat{V})$. As $\Meet_{V\in\scr{V}}clV\sse\Meet_{V\in\scr{V}}c(\widehat{V})$ by Proposition~\ref{properties}(c), it follows that
$$\{x\}=\Meet\scr{V}=\Meet_{V\in\scr{V}}clV=\Meet_{V\in\scr{V}}c(\widehat{V}).$$
\end{proof}

\begin{proposition}\label{cbound}
If $X$ is Hausdorff and $A\sse X$, then $|c(A)|\leq |A|^{t_c(X)\psi_c(X)}$.
\end{proposition}

\begin{proof}
Let $\kappa=t_c(X)\psi_c(X)$. For each $x\in c(A)$, by Proposition~\ref{intersection} there exists a family $\scr{V}_x$ of open sets such that $|\scr{V}_x|\leq\kappa$ and
$$\{x\}=\Meet\scr{V}_x=\Meet_{V\in\scr{V}_x}clV=\Meet_{V\in\scr{V}_x}c(\widehat{V}).$$
As $t_c(X)\leq\kappa$, for all $x\in c(A)$ there exists $A(x)\in[A]^{\leq\kappa}$ such that $x\in c(A(x))$.

Define $\phi:c(A)\to\left[[A]^{\leq\kappa}\right]^{\leq\kappa}$ by 
$$\phi(x)=\{\widehat{V}\meet A(x):V\in\scr{V}_x\}.$$
Observe that $\phi(x)\in\left[[A]^{\leq\kappa}\right]^{\leq\kappa}$. Fix $x\in c(A)$. We will show that $x\in c(\widehat{V}\meet A(x))$ for all $V\in\scr{V}_x$. Let $V\in\scr{V}_x$ and let $U$ be any open set containing $x$. As $x\in c(A(x))$, there exists $a\in A(x)$ such that $U\meet V\in\scr{U}_a$. Thus, $a\in\widehat{U\meet V}=\widehat{U}\meet\widehat{V}$ and it follows that $\widehat{U}\meet\widehat{V}\meet a(X)\neq\es$. This shows $x\in c(\widehat{V}\meet A(x))$. Thus, 
$$\{x\}\sse\Meet_{V\in\scr{V}_x}c(\widehat{V}\meet A(x))\sse\Meet_{V\in\scr{V}_x}c(\widehat{V})=\{x\},$$
where the second containment above follows from Proposition~\ref{properties}(a). Then $\{x\}=\Meet_{V\in\scr{V}_x}c(\widehat{V}\meet A(x))$. Thus if $x\neq y$ then $\phi(x)\neq\phi(y)$, and $\phi$ is one-to-one. Therefore, $|c(A)|\leq |A|^{\kappa}$.
\end{proof}

We turn now to our main result, a new bound for the cardinality of a Hausdorff space $X$. To establish this bound, we use the set-theoretic Theorem 3.1 from \cite{Hodel2006}. This theorem generalizes many closing-off arguments needed to prove cardinality bounds on topological spaces. For reference, we re-state the particular case of this theorem that is used here. 

\begin{theorem}[Hodel]\label{hodel}
Let $X$ be a set, $\kappa$ be an infinite cardinal, $d:\scr{P}(X)\to\scr{P}(X)$ an operator on $X$, and for each $x\in X$ let $\{V(\alpha,x):\alpha<\kappa\}$ be a collection of subsets of $X$. Assume the following:
\begin{itemize}
\item[(T)] (tightness condition) if $x\in d(H)$ then there exists $A\sse H$ with $|A|\leq\kappa$ such that $x\in d(A)$;
\item[(C)] (cardinality condition) if $A\sse X$ with $|A|\leq\kappa$, then $|d(A)|\leq 2^\kappa$;
\item[(C-S)] (cover-separation condition) if $H\neq\es$, $d(H)\sse H$, and $q\notin H$, then there exists $A\sse H$ with $|A|\leq\kappa$ and a function $f:A\to\kappa$ such that $H\sse\Un_{x\in A}V(f(x),x)$ and $q\notin\Un_{x\in A}V(f(x),x)$.
\end{itemize}
Then $|X|\leq 2^\kappa$.
\end{theorem}

Typically, the operator $d$ used in Theorem~\ref{hodel} is either the standard closure operator $cl$, or in some instances the $\theta$-closure $cl_\theta$. We use the operator $c$.

\begin{theorem}\label{newbound}
If $X$ is Hausdorff then $|X|\leq 2^{aL^\prime(X)t_c(X)\psi_c(X)}$.
\end{theorem}

\begin{proof}
Let $\kappa=aL^\prime(X)t_c(X)\psi_c(X)$. As $\psi_c(X)\leq\kappa$, for all $x\in X$ there exists a family $\scr{W}_x=\{W(\alpha,x):\alpha<\kappa\}$ of open sets such that $\{x\}=\Meet\scr{W}_x=\Meet_{W\in\scr{W}_x}clW$. For all $x\in X$ and $\alpha<\kappa$, set $V(\alpha,x)=cl(W(\alpha,x))$. We verify the three conditions in Theorem~\ref{hodel}, where the operator $d$ is $c$. The (T) condition follows immediately as $t_c(X)\leq\kappa$, and (C) follows from Proposition~\ref{cbound}. To verify (C-S), suppose $H\neq\es$ satisfies $c(H)\sse H$. Then, $c(H)=H$, as $H\sse c(H)$ by Proposition~\ref{properties}(a), and $H$ is $c$-closed. Let $q\notin H$. For all $a\in H$, there exist $\alpha_a<\kappa$ such that $q\notin cl(W(\alpha_a,a))=V(\alpha_a,a)$. Define $f:A\to\kappa$ by $f(a)=\alpha_a$. Then $\{W(f(a),a):a\in H\}$ is a cover of $H$ by sets open in $X$. As $aL^\prime(X)\leq\kappa$ and $H$ is $c$-closed, there exists $A\in[H]^{\leq\kappa}$ such that $H\sse\Un_{a\in A}V(f(a),a)$. Since $q\notin\Un_{a\in A}V(f(a),a)$, we see that (C-S) is satisfied. By Theorem~\ref{hodel}, $|X|\leq 2^\kappa$.
\end{proof}

As $aL^\prime(X)\leq aL_c(X)$ by Proposition~\ref{anotherProp} and $t_c(X)\psi_c(X)\leq\chi(X)$ by Proposition~\ref{tcx}, we obtain the following Corollary~\ref{belcam}. This is a slight weakening of the Bella-Cammaroto bound $2^{aL_c(X)t(X)\psi_c(X)}$ for Hausdorff spaces. While $aL^\prime(X)\leq aL_c(X)$, it is unclear whether $t(X)$ and $t_c(X)$ are comparable for a non-regular space $X$, making it unclear whether $2^{aL^\prime(X)t_c(X)\psi_c(X)}$ and $2^{aL_c(X)t(X)\psi_c(X)}$ are comparable.

\begin{corollary}\label{belcam}[Bella/Cammaroto]
If $X$ is Hausdorff then $|X|\leq 2^{aL_c(X)\chi(X)}$.
\end{corollary}

\begin{corollary}\label{dowporter}[Dow/Porter]
If $X$ is H-closed then $|X|\leq 2^{\psi_c(X)}$.
\end{corollary}
\begin{proof}
The semiregularization $X(s)$ of $X$ is also H-closed, and so by Corollary~\ref{aLL} and Corollary~\ref{hclosed} it follows that $aL^\prime(X_s)=\aleph_0$. Thus, by Theorem~\ref{newbound}, we have that $$|X|=|X_s|\leq 2^{t_c(X(s))\psi_c(X(s))}\leq 2^{\chi(X(s))}=2^{\psi_c(X(s))}\leq 2^{\psi_c(X)},$$
where the second equality follows as $X(s)$ is minimal Hausdorff.
\end{proof}

We see then that Theorem~\ref{newbound} leads to a common proof of the cardinality bound $2^{\chi(X)}$ for both H-closed spaces and Lindel\"of spaces simultaneously. We can isolate the precise property $\scr{P}$ that both H-closed spaces and Lindel\"of spaces $X$ share from which it follows from Theorem~\ref{newbound} that $|X|\leq 2^{\chi(X)}$. Property $\scr{P}$ is the property that every open cover $\scr{V}$ of $X$ has a countable subfamily $\scr{W}$ such that $X=\Un_{W\in\scr{W}}\widehat{W}$. That is, $L^\prime(X)=\aleph_0$. In fact, the weaker property $aL^\prime(X)=\aleph_0$ also suffices.

\section{Generalized H-closed Spaces}

The standard method of generalizing the concept of  H-closed is to use the well-known cardinality invariant of almost Lindel\" of -- when  $aL(X) \leq \aleph_0$, the space $X$ is a generalized H-closed space.  One of the main goals in this paper is seek generalized H-closed spaces for which it is possible to obtain a cardinality bound of $X$. A space $X$ satisfying $a'L(X) \leq \aleph_0$ is another generalized H-closed space for which it is possible to obtain a cardinality bound of $X$ (see Theorem~\ref{newbound}).  We used the concept of $\kappa$-H-closed, another generalized H-closed space, to obtain a cardinality bound of H-closed spaces (see Theorem~\ref{kappafilter}).  In this section, we examine three new generalized H-closed concepts with the common goal of obtaining a cardinality bound of a space.\\\\
\emph{Approach I.}\\

The roots of our first generalized H-closed space can be traced back  to  the famous 1929 memoir (the Russian version of~\cite{Ale-Ury}) and uses a recent characterization of H-closed spaces by Osipov~\cite{O}.  Alexandroff and Urysohn proved this property of H-closed spaces: If $X$ is an H-closed space and $A \subseteq X$ is an infinite subset, there is a point $p \in X$ such that $|A| = |A \cap cl(U)|$ whenever $p \in U \in \tau(X)$.  

\begin{definition} Let $X$ be a space and  $A \subseteq X$. A point $p \in X$ is a $
\Theta$-complete accumulation point of $A$ (we write $p \in \Theta$CAP($A$)) if whenever $p \in U \in \tau(X)$, $|cl(U) \cap A| = |A|$.  In particular, $cl_{\theta}A = (\Theta$CAP($A$))$ \cup  (cl_{\theta}A\backslash\Theta$CAP($A$)).
\end{definition}

An exciting new characterization of H-closed spaces using the concept of $\Theta$-complete accumulation points was established in 2013  by Osipov~\cite{O}.

\begin{theorem}\label{osipov}{\rm{[Osipov]}} Let $X$ be a Hausdorff space and $A \sse X$ an infinite subset,
\begin{itemize}
\item[(a)] $X$ is H-closed iff for each  open cover $\mathcal{C}$ of $\Theta$CAP($A$), there is a finite subfamily $\scr{F} \subseteq\scr{C}$ such that    $|A \backslash int(cl(\cup\scr{F}))| < |A|$. 
\item[(b)] If $X$ is H-closed space , then  $\Theta$CAP($A$) is an H-set. 
\end{itemize} 
\end{theorem}

We start the process of generalizing H-closed spaces by expanding the notation $\Theta$CAP. 

\begin{definition} Let $X$ be an H-closed space, $\ka$ an infinite cardinal, and $A$ an infinite subset. Let $cl_{\theta }^{\ka}A$ denote $\{p \in X: |cl(U) \cap A| \geq \ka$ for $p \in U \in \tau(X)\}$.  Note that $cl_{\theta }^{|A|}A  = \Theta$CAP($A$).  If $\sigma$ is an infinite cardinal and  $\kappa \geq \sigma$, then $cl^{\kappa}_{\theta}(A) \subseteq cl^{\sigma}_{\theta}(A)$  and $cl^{\omega}_{\theta}(A)$ is the set of accumulation points of $A$.
\end{definition}

Using a techniques similar to the proof of Theorem~\ref{osipov}, it is possible to obtain this result at the $\kappa$ level.:

\begin{proposition}\label{osipovk}  Let $X$ be an H-closed space, $\ka$ an infinite cardinal, $A$ an infinite subset,  and $\scr{C}$  an open cover of $cl_{\theta }^{\ka}A$.  There is a finite subfamily $\scr{D} \subseteq \scr{C}$ such that    $|A \backslash int(cl(\cup\scr{D}))| < \ka$. 
\end{proposition}

\begin{proof} Let $\scr{C}$ be an open cover of $cl_{\theta }^{\ka}A$.  For $p\not\in cl_{\theta }^{\ka}A$, there is an open set $U_p$ such that $p \in U_p$ and $|cl(U_p) \cap A| < \ka$.  Let $\scr{O} = \{U_p: p \not\in cl_{\theta }^{\ka}A\}$.
As $X$ is H-closed, there are finite subfamilies $\scr{D} \subseteq\scr{C}$ and $\scr{U} \subseteq\scr{O}$ such that $X = cl(\cup\scr{D}) \cup cl(\cup\scr{U})$.  So, $cl(X\backslash cl(\cup\scr{D})) \subseteq cl(\cup\scr{U})$ implying that $X\backslash int(cl(\cup\scr{D})) \subseteq cl(\cup\scr{U})$.  Since $|(cl(\cup\scr{U})) \cap A| < \ka$, it follows that $|A\backslash int(cl(\cup\scr{D}))| < \ka$. \end{proof}

\begin{definition}\label{kwHclosed}
Let $X$ be a space and $\kappa$ be infinite cardinal.  A  filter base $\scr{F}$ on $X$ is said to be $\kappa-${\it wide} if $|cl_{\theta}(A)| \geq \kappa$ for each $A \in\scr{F}$. A space $X$ is $\kappa$\emph{wH-closed} if for each $\kappa-$wide filter base $\scr{F}$ on $X$, $a_{\theta}\scr{F}\ne \varnothing$.
\end{definition}
 
\begin{proposition}\label{propA}
 Let $X$ be a space and $\kappa$ be an infinite cardinal.
\begin{itemize}
\item[(a)] The space $X$ is H-closed iff $X$ is $\aleph_0$wH-closed.
\item[(b)]  Let $X$ be a space and $\kappa$ be infinite cardinal.  $X$ is $\kappa$wH-closed iff for each $\kappa-$wide open filter base $\scr{F}$, $a\scr{F}\ne \varnothing$.

\end{itemize}
\end{proposition}

\begin{proof}
The proof of (a) is immediate. The proof of (b) follows the known result that if $\scr{F}$ is a filter base on a space $X$, then the open filter base $\scr{G} = \{U \in \tau(X): F \subseteq U$ for some $F \in \scr{F}\}$ has the property $a\scr{G} = a_{\theta}\scr{F}$.
\end{proof}

\begin{theorem}\label{kwHchar} 
Let $X$ be space and $\kappa$ an infinite cardinal.  The space $X$ is 
$\kappa$wH-closed iff for every subset $A \subseteq X$ where $\kappa \leq |A|$ and $\scr{C}$ is an open cover of $cl_{\theta }^{\ka}A $, there is a finite subfamily $\scr{B} \subseteq\scr{C}$ such that $|A\backslash int(cl(\cup\scr{B}))| < \kappa$.  
\end{theorem}

\begin{proof} 
Suppose $X$ is $\kappa$wH-closed and  $A \subseteq X$ where $\kappa \leq |A|$ and $\scr{C}$ is an open cover of $cl_{\theta }^{\ka}A $.  For each $p \not\in cl_{\theta }^{\ka}A $, there is $p \in U_p \in\tau(X)$ such that  $|cl(U_p) \cap A| < \kappa$.  Let $\scr{E} = \{ U_p: p \not\in cl_{\theta }^{\ka}A \}$.   Assume, by way of contradiction, that for each  finite subfamily $\scr{B}$ of $\scr{C}$,  $|A\backslash int(cl(\cup\scr{B}))| \geq \kappa$.\\\\
\emph{Claim}: $\scr{F} = \{X\backslash cl(\cup\scr A): \scr A \in  [\scr{E} \cup \scr{C}]^{<\omega}\}$ is an $\kappa$-wide filter base such that $|cl(V)| \geq \kappa$ for each $V \in \scr{F}$.

\begin{proof}[Proof of Claim] Let $\scr A \in  [\scr{C}\cup \scr{E}]^{<\omega}$.  Then there are  finite subfamilies $\scr{B}\subseteq \scr{C}$ and  $\scr{D} \subseteq \scr{E}$ such that $\scr A = (\cup\scr{B}) \cup (\cup\scr{D})$.  It suffices to show that $|cl(X\backslash cl(\cup \scr A))| = |cl(X\backslash cl( (\cup\scr{B}) \cup (\cup\scr{D})))| \geq \kappa$.  Note that $|cl(\cup\scr{D})\cap A| < \kappa$ implies $|(X\backslash cl(\cup\scr{D}) )\cap A| = |A\backslash cl(\cup D)| = |A| \geq \kappa$. Also, note that $|cl(X\backslash cl(\cup\scr{B}))| \geq \kappa$ by the assumption.  We have:  
\begin{align}
cl(X\backslash cl( (\cup\scr{B} \cup (\cup\scr{D})))&= cl((X\backslash cl( \cup\scr{B})) \cap (X\backslash cl( \cup\scr{D})))\notag\\
&= cl(cl(X\backslash cl( \cup\scr{B})) \cap (X\backslash cl( \cup\scr{D})))\notag\\
&= cl((X\backslash int(cl( \cup\scr{B}))) \cap (X\backslash cl( \cup\scr{D})))\notag\\
&\supseteq cl((X\backslash int(cl( \cup\scr{B}))) \cap (A\backslash cl( \cup\scr{D}))) \notag\\
&= cl((A\backslash int(cl( \cup\scr{B}))) \backslash cl( \cup\scr{D}))). \notag
\end{align}
\noindent This shows that $|cl(X\backslash cl(\cup \scr A))| \geq |cl((A\backslash int(cl( \cup\scr{B}))) \backslash cl( \cup\scr{D})))| \geq \kappa$ and   $\scr{F}$ is a  $\kappa-$wide filter base.  
\end{proof}
As $X$ is $\kappa$wH-closed, there is some  $p \in a\scr{F}$.  If $p \in cl_{\theta }^{\ka}A $, there is $U \in \scr{C}$ such that $p \in U$.  Thus, $X\backslash cl(U) \in\scr{F}$ and $p \not \in cl(X\backslash cl(U))$; so, $p \not\in a(\scr{F})$.  On the other hand, if $p \not\in a\scr{F}$, there is $U \in\scr{E}$ such that $p \in U$.  Again, $X\backslash cl(U) \in \scr{F}$, $p \not \in cl(X\backslash cl(U))$, and $p \not\in a(\scr{F})$.  Hence, $a(\scr{F}) = \varnothing$.  This contradicts  the hypothesis.

To show the converse, let $\scr{F}$ be a free $\kappa-$wide open filter base on $X$.  Let $U \in\scr{F}$ such $|cl(U)|$ is minimum. We will apply the condition in the statement of the theorem to the set $cl(U)$. In particular, $|cl(U)| \geq \kappa$.  If $p \in cl_{\theta }^{\ka}(cl(U)) $, there is $V_p \in \scr{F}$ such that $p \not\in cl(V_p)$ and $V_p \subseteq U$.  Then $p \in X\backslash cl(V_p)$ and $|cl(X\backslash cl(V_p)) \cap cl(U)| \geq \kappa$. Let $\scr{C}=\{X\backslash cl(V_p): p \in cl_{\theta }^{\ka}(cl(U))\}$.  By the hypothesis of the converse, there is a finite subfamily $\scr{B}\subseteq\scr{C}$ such that $|cl(U)\backslash int(cl(\cup\scr{B}))| < \kappa$.  For $V = \cap\{V_p: X\backslash cl(V_p) \in\scr{B}\} \cap U$, note that $V \in \scr F$,  $cl(V) \subseteq cl(U)$ and $V \cap (\cup\scr{B}) = \varnothing$. It follows that $clV \cap int(cl(\cup\scr{B})) = \varnothing$.  This shows that $clV \subseteq clU \backslash int(cl(\cup\scr{B}))$. That is, $V \in \scr{F}$ and $|clV| < \kappa$, a contradiction. 
\end{proof}

As corollaries of Theorems~\ref{osipov} and \ref{kwHchar}, we have the following results.

\begin{corollary} Let $X$ be a space and $\kappa$ be an infinite cardinal. 
\begin{itemize}
\item[(a)] The space $X$ is H-closed iff $X$ is $\kappa$wH-closed for all infinite $\kappa \leq |X|$.
\item[(b)] If  $X$  is 
a $\kappa$wH-closed space, $A \subseteq X$ such that $|A| \geq \kappa|$, and $\scr{F}$ is a $\kappa-$wide open filter base on $X$ that meets  $cl_{\theta }^{\ka}A $, then  $a(\scr{F}) \cap cl_{\theta }^{\ka}A \ne \varnothing$.
\item[(c)] If $X$ is $\kappa$wH-closed and $A \subseteq X$ such that $|A| \geq \kappa$, then $cl^{\kappa}_{\theta}A \ne \varnothing.$
\end{itemize}
\end{corollary}

\begin{proof} 
The proof of (a) is straightforward. To prove (b), let $\scr{F}$ be a $\kappa-$wide open filter base on $X$ that meets $cl_{\theta }^{\ka}A$. Assume that $a(\scr{F}) \cap cl_{\theta }^{\ka}A = \varnothing$. 
Let $U \in \scr{F}$ such $|cl(U)\cap A|$ is minimum.  Note that  $|cl(U)\cap A| \geq \kappa$.
 If $p \in cl_{\theta }^{\ka}(A) $, there is $V_p \in \scr{F}$ such that $p \not\in cl(V_p)$ and $V_p \subseteq U$.  Then $p \in X\backslash cl(V_p)$ and $|cl(X\backslash cl(V_p)) \cap A| \geq \kappa$. Let $\scr{C}=\{X\backslash cl(V_p): p \in cl_{\theta }^{\ka}A\}$.  
 As $X$ is $\kappa$wH-closed, there is a finite subfamily $\scr{B}\subseteq\scr{C}$ such that $|A\backslash int(cl(\cup\scr{B}))| < \kappa$.  
 Now $V = \cap\{V_p: X\backslash cl(V_p) \in \scr{B}\} \cap U \in \scr{F}$.  Also, $clV \subseteq clU$ and $V \cap (\cup \scr{B}) = \varnothing$. It follows that $clV \cap int(cl(\cup\scr{B})) = \varnothing$.  
 This shows that $clV \subseteq clU \backslash int(cl(\cup\scr{B}))$.  Thus,  $clV \cap A \subseteq (clU \backslash int(cl(\cup\scr{B}))) \cap A = (clU \cap A)\backslash int(cl(\cup\scr{B}))  \subseteq  A\backslash int(cl(\cup\scr{B})).$ This implies there is a  $V \in\scr{F}$ such that $|clV\cap A| < \kappa$, a contradiction. 
 
To show (c), assume that $cl^{\kappa}_{\theta}A = \varnothing.$  For each $ p \in X$, there is an open set $p \in U_p \in \tau(X)$ such that $|clU_p \cap A| < \kappa$.  Then $\mathcal U = \{U_p:p \in X\}$ is an open cover of $X$ (and $cl^{\kappa}_{\theta}A).$  There is a finite $\mathcal V \subseteq \mathcal U$ such that $|A\backslash int(cl(\cup \mathcal V))| < \kappa$.  Let $B = A\backslash int(cl(\cup \mathcal V))$.  Then $A \subseteq   B \cup int(cl(\cup \mathcal V)) \subseteq B \cup cl(\cup \mathcal V)$.  It follows that $A  \subseteq B \cup (\cup_{V \in \mathcal V} cl(V) \cap A)$.  Thus, $|A| \leq |B| + \Sigma_{V \in \mathcal V}|cl(V) \cap A)| < \kappa$.  This is a contradiction. 
 \end{proof}

The study of $\kappa$wH-closed spaces  is a new approach to understanding the theory of  H-closed spaces by using  the width of a filter base.  The width  is a measure of the size of the closure of the elements of a filter base.  
 
There is still the question of obtaining a cardinality bound of $\kappa$wH-closed spaces.  We are able to obtain such a result only for $\kappa$wH-closed spaces with a dense subset of isolated points. We start by statinga well-known result that is similar to \ref{cbound}.

\begin{lemma}\label{clk} Let $\kappa$ be an infinite cardinal and $\chi(X) \leq \kappa$.  For $A \subseteq X$, $|cl(A)| \leq |A|^{\kappa}$.
\end{lemma}

\begin{theorem}\label{kwHclosedDense}
Let $\kappa$ be an infinite cardinal and  $X$ a $\kappa$wH-closed space with a dense set of isolated points and $\chi(X) \leq \kappa$.   
Then $|X| \leq 2^{\kappa}$.
\end{theorem}

\begin{proof}
For each $p \in X$, let $V(p) = \{V(\al, p): \alpha < \kappa\}$ be an open neighborhood base at  $p$, and for $B \subseteq X$ and $f:B \rightarrow \kappa$, let $V(f,B) = \bigcup_{p \in B}V(f(p),p)$.  Let $D$ be the set of isolated points of $X$. If $p \in D$, we let $V(\al,p) = \{p\}$ for all $\al \in \ka$.

Let $H:\scr{P}(D) \rightarrow D$ be a choice function and $A_0 = {H(\varnothing)}$. We will inductively define $A_{\alpha}$ for $\alpha < \kappa^+$. For $\alpha < \kappa^{+}$, suppose $A_{\beta}$ is defined for $\beta < \alpha$. 
Let $$A_{\beta}  =   \bigcup_{\alpha < \beta}A_{\alpha}  \cup  \bigcup\{H(D\backslash V(g,B)): B \in [cl(\bigcup_{\alpha < \beta}A_{\alpha}) ]^{\leq\kappa}, g:B\rightarrow\kappa \}.$$ 

By induction, $|A_{\alpha}| \leq 2^{\kappa}$; let $C = \bigcup_{\alpha < \kappa^+}A_{\alpha}$.  It  follows that  $|C| \leq 2^{\kappa}$.  By Lemma~\ref{clk}, for $A = cl(C)$, $|A| \leq 2^{\kappa}$.  As $C \subseteq D$ and is open, we also have that $A = cl_{\theta}(C)$.  Also, note that as $C$ is an increasing chain over $\kappa^+$ and $\chi(X) \leq \kappa$,  $A = cl(C) =  \bigcup_{\alpha < \kappa^+}cl(A_{\alpha})$.
 
To apply Theorem~\ref{kwHchar},  we need to show that $|C| \geq \kappa$.  Suppose that 
 $|C| \leq \kappa$. As $C = \bigcup_{\alpha < \kappa^+}A_{\alpha}$, there is $\alpha < \kappa$ such that $C \subseteq A_{\alpha}$.  Let $g:C \rightarrow \kappa: p\mapsto 0$.  Then $C = V(g,C)$ and it follows that $D\backslash V(C,g) = \varnothing$.  Thus, $D \subseteq C$, $|D| \leq \kappa$, and it follows that $|X| \leq 2^{\kappa}$ and we are done.  We are reduced the case when $|C| \geq \kappa$.
   
To finish the proof of the theorem, we will prove that $X = cl(C)=A$ by showing that $D \subseteq C$.  Let $d \in D\backslash C = D\backslash A$. For each $p \in A$, there is some $U_p = V(\al_p,p) \in V(p)$ such that $d \not\in cl(U_p)$.  Then $\scr{C}= \{U_p: p \in A\}$ is an open cover of $A = cl_{\theta}C \supseteq  cl_{\theta}^{\kappa}C $; note that $cl_{\theta}^{\kappa}C \sse X\backslash D$.
 By Theorem~\ref{kwHchar}, there is a finite subfamily $\scr{B}$ of $\scr{C}$ such that $|C\backslash int(cl(\cup\scr{B}))| < \kappa$. There is a finite subset $F \subseteq A$ such that $\scr{B} = \{U_p: p \in F\}$.  For $U = \cup_{p \in F}U_p\cap D$, $C \backslash  int(cl(\cup\scr{B})) = C \backslash U$.  Define $h: C\backslash U \cup F \rightarrow \kappa: p\mapsto \al_p$; it follows that $C \subseteq V(h,C\backslash U \cup F)$.  
 
Now, as $|C\backslash U \cup F| < \kappa$ and $cl(C) =   \bigcup_{\alpha < \kappa^+}cl(A_{\alpha})$, there is some $\beta <\kappa^+$, such that $C\backslash U \cup F \subseteq cl(A_{\beta})$.  As $d \not\in V(h,C\backslash U \cup F)$, $D\backslash V(h,C\backslash U \cup F) \ne \varnothing$.  Thus, $H( D\backslash V(C\backslash U \cup F,h)) \in A_{\beta +1} \subseteq C$, a contradiction.  This completes the proof that $D \subseteq C$ and finishes the proof. 
\end{proof}
We ask whether the above Theorem~\ref{kwHclosedDense} is true without the hypothesis that $X$ has a dense set of isolated points.
\begin{question}
If $\kappa$ is an infinite cardinal and $X$ a $\kappa$wH-closed space such that $\chi(X)\leq\kappa$, is $|X|\leq 2^{\kappa}$?
\end{question}

\noindent\emph{Approach II.}\\
    
The application of $\kappa$-H-closed spaces in Theorem~\ref{kappafilter} to obtain a cardinality bound of H-closed spaces  provides another approach to studying H-closed spaces by using ``thin'' filter bases where a filter base is a member of $[[X]^{\leq \kappa}]^{\leq \kappa}$.   This technique is another way of measuring the width of a filter base and provides our second path in defining generalized H-closed.  

\begin{definition}\label{defII}  
Let $X$ be a space, $\kappa$ an infinite cardinal, and  $A \subseteq X$ such that $|A| \geq \kappa$. Define $c^{\ka}(A) = \{x \in X: \text{ if } x \in U \in \tau(X), \text{then } |\hat{U} \cap A| \geq \kappa \}$.  A space $X$ is $\kappa H^{\prime}$-\emph{closed} if $A \subseteq X$, $|A| \geq \kappa$, and $\mathcal U$ is an open cover of $c^{\ka}(A)$, there is a finite subfamily $\mathcal V \subseteq \mathcal U$ such that $|A \backslash \cup_{V \in \mathcal V}\hat{V}| < \kappa$.  If $|X| < \kappa$, it follows that $X$ is $\ka$H$^{\prime}$-closed.
\end{definition}

\begin{proposition}\label{HclII}  
A space $X$ is $\aleph_0$H$^{\prime}$-closed iff $X$ is H-closed.
\end{proposition}

\begin{proof}
Suppose $X$ is $\aleph_0$H$^{\prime}$-closed.  Let $\mathcal U$ be open cover of $X$.  We can assume that $|X| \geq \aleph_0$.  Then $\mathcal U$ covers $c^{\aleph_0}(X)$. Then there is a finite $\mathcal V \subseteq \mathcal U$ such that $|X\backslash \cup_{\mathcal V}\hat{V}| < \aleph_0$.  By Theorem~\ref{HclosedChar}, $X$ is H-closed.  Conversely, suppose $X$ is H-closed. Let $A \subseteq X$ such that $|A|\geq \aleph_0$.  Let $\mathcal U$ be open cover of $c^{\aleph_0}(A)$.  For each $p \not\in c^{\aleph_0}(A)$, there is $p \in U_p \in \tau(X)$ such that $|\hat{U_p} \cap A| < \aleph_0$.  Now, $\{U_p: p \not\in c^{\aleph_0}(A)\} \cup \mathcal U$ is open cover of $X$.  There is a finite $B \subseteq X\backslash c^{\aleph_0}(A)$ and a finite $\mathcal V \subseteq \mathcal U$ such that $X = \cup_{p \in B}\hat{U_p} \cup \cup_{V \in  \mathcal V}\hat{V}$.  Thus, $X \backslash \cup_{V \in  \mathcal V}\hat{V} \subseteq \cup_{p \in B}\hat{U_p}$ and $|A \backslash \cup_{V \in  \mathcal V}\hat{V} | \leq |\cup_{p \in B}(\hat{U_p}\cap A)| \leq \sum_{p\in B}|\hat{U_p}\cap A| < \aleph_0$.   This shows that $X$ is $\aleph_0$H$^{\prime}$-closed. 
\end{proof}

\begin{proposition}\label{cardII} 
Let $\kappa$ be infinite cardinal and $X$ be  $\kappa$H$^{\prime}$-closed. Then $aL'(X) < \kappa$.
\end{proposition}

\begin{proof} Let $A$ be $c$-closed.  If $|A| < \kappa$, then $aL'(A,X) < \kappa.$  So, suppose that $|A| \geq \kappa$.  Let $\mathcal U$ be open cover of $A$. For each $p \not\in A$, there is $p \in U_p \in \tau(X)$ such that $\hat{U_p} \cap A = \varnothing$.  Now, $\{U_p: p \not\in A\} \cup \mathcal U$ is open cover of $X$.  There is a finite $B \subseteq X\backslash A$ and a finite $\mathcal V \subseteq \mathcal U$ such that $|X\backslash(\cup_{p \in B}\hat{U_p} \cup \cup_{V \in  \mathcal V}\hat{V})| < \kappa$.  Now, $|A\backslash(\cup_{p \in B}\hat{U_p} \cup \cup_{V \in  \mathcal V}\hat{V})| = |A\backslash( \cup_{V \in  \mathcal V}\hat{V})|  < \kappa$.  Thus, there is $\mathcal W \subseteq \mathcal U$ such that $|\mathcal W| < \kappa$ and $A\backslash(\cup_{V \in  \mathcal V}\hat{V}) \subseteq  \cup_{\mathcal W}W)   \subseteq  \cup_{\mathcal W}\hat{W}) $.  Therefore, $A \subseteq \cup_{\mathcal V}\hat{V} \cup\cup_{\mathcal W}\hat{W}$ and $|\mathcal V \cup \mathcal W| < \kappa$.  So, 
$aL'(A,X) < \kappa$. 
\end{proof} 

\begin{definition}  
For a space $X$, define $z(X) = inf\{\kappa \geq \aleph_0: X \text{ is } \kappa\text{H}'\text{--closed}\}$.
\end{definition}

By Proposition~\ref{cardII} and Theorem~\ref{newbound}, we have the following two results.

\begin{corollary}\label{corII} For a space $X$, 
\begin{itemize}
\item[(a)] $aL'(X)^+ \leq z(X)$ and 
\item[(b)] $|X| \leq 2^{z(X)t_c(X)\psi_c(X)}$.
\end{itemize}
\end{corollary}

\begin{remark}
This last corollary is an indication that the concept of $\kappa$H$^{\prime}$-closed is subsummed by the theory using $c$-closure and $aL'$.
\end{remark}

\noindent\emph{Approach III.}\\

In this third approach to generalized H-closed spaces, we modify the concept  defined in Definition~\ref{defII} in Path II.

\begin{definition}\label{kH''closed}
Let $X$ be a space, $\kappa$ an infinite cardinal, and  $A \subseteq X$ such that $|A| \geq \kappa$. Define $c^{\ka}(A) = \{x \in X: \text{ if } x \in U \in \tau(X), \text{then } |\hat{U} \cap A| \geq \kappa \}$.  A space $X$ is $\kappa H^{\prime\prime}$-\emph{closed} if $A \subseteq X$, $|A| \geq \kappa$, and $\mathcal U$ is an open cover of $c^{\ka}(A)$, there is a subfamily $\mathcal V \subseteq \mathcal U$ such that $|\mathcal V| \leq \kappa$ and $|A \backslash \cup_{V \in \mathcal V}\hat{V}| < \kappa$.  In particular, if $|X| < \kappa$, then $X$ is $\ka$H$^{\pp}$-closed.

\end{definition}
  
Using essentially the same proof as the proof of Proposition~\ref{cardII}, we obtain the following result.
  
\begin{proposition}\label{cardIII} 
Let $\kappa$ be infinite cardinal and $X$ be  $\kappa$H$^{\pp}$-closed.  Then $aL'(X) \leq \kappa$.
\end{proposition}

\begin{definition}   
For a space $X$, define $z^{\pp}(X) = inf\{\kappa \geq \aleph_0: X \text{ is } \kappa\text{H}''\text{--closed}\}$.
\end{definition}

By Proposition~\ref{cardIII} and Theorem~\ref{newbound}, we have the following two results.

\begin{corollary}\label{corIII} 
For a space $X$, 
\begin{itemize}
\item[(a)] $aL'(X) \leq z^{\pp}(X)$ and 
\item[(b)] $|X| \leq 2^{z^{\pp}(X)t_c(X)\psi_c(X)}$.
\end{itemize}
\end{corollary}

\begin{remark}
Using Approach III, \ref{corIII}(a) is sharper than \ref{corII}(a).  However, the price is that the counterpart to Proposition~\ref{HclII} is not true; that is, in the case when $\kappa = \aleph_0$, we do not necessarily get H-closed. In fact, any countable space is  $\aleph_0$H$^{\pp}$-closed.
\end{remark}

\end{document}